\documentclass{amsart}
\usepackage{amsmath}
\usepackage{amssymb}
\usepackage{amsthm}
\usepackage{amscd}
\usepackage{stmaryrd}
\usepackage[dvipdfmx]{graphicx}
\usepackage[all]{xy}
\usepackage{wrapfig}
\usepackage{color}
\usepackage{latexsym}
\usepackage{enumerate}
\usepackage{cases}
\usepackage[top=25mm,bottom=25mm,left=30mm,right=30mm]{geometry}
\usepackage{graphics}
\usepackage[dvipdfmx]{graphicx}
\usepackage{tikz}
\usetikzlibrary{patterns}
\usepackage{pxpgfmark}
\usepackage{multirow}
\usepackage[foot]{amsaddr}
\usepackage{comment}
  \newlength\squareheight
  \newcommand\squareslash{\tikz{\draw (0,0) rectangle (\squareheight,\squareheight);\draw(0,0) -- (\squareheight,\squareheight)}}
\makeatletter
\newcommand{\xequal}[2][]{\ext@arrow 0055{\equalfill@}{#1}{#2}}
\def\equalfill@{\arrowfill@\Relbar\Relbar\Relbar}
\makeatother
\usetikzlibrary{intersections, calc,decorations.markings}
\numberwithin{equation}{section}
\theoremstyle{definition}
\newtheorem{example}{Example}[section]
\newtheorem{definition}[example]{Definition}

\theoremstyle{plain}
\newtheorem{lemma}[example]{Lemma}
\newtheorem{theorem}[example]{Theorem}
\newtheorem{proposition}[example]{Proposition}
\newtheorem{corollary}[example]{Corollary}

\DeclareMathOperator{\add}{{\rm add}}
\DeclareMathOperator{\proj}{{\rm proj}}

\DeclareMathOperator{\Mod}{{\rm mod}}

\DeclareMathOperator{\Tr}{{\rm Tr}}
\DeclareMathOperator{\sH}{\mathsf{H}}
\DeclareMathOperator{\X}{\mathsf{X}}
\DeclareMathOperator{\e}{\mathsf{e}}

\DeclareMathOperator{\Hom}{{\rm Hom}}
\DeclareMathOperator{\End}{{\rm End}}

\DeclareMathOperator{\sttilt}{{\rm s \tau \mathchar`-tilt}}
\DeclareMathOperator{\trigid}{{\rm i \tau \mathchar`-rigid}}

\DeclareMathOperator{\ctilt}{{\rm c \mathchar`-tilt}}
\DeclareMathOperator{\rigid}{{\rm irigid}}

\DeclareMathOperator{\clv}{{\rm cl-var}}
\DeclareMathOperator{\clus}{{\rm cluster}}

\DeclareMathOperator{\x}{\bf{x}}

\DeclareMathOperator{\bA}{\mathbb{A}}

\DeclareMathOperator{\bR}{\mathbb{R}}
\DeclareMathOperator{\bT}{\mathbb{T}}
\DeclareMathOperator{\bZ}{\mathbb{Z}}

\DeclareMathOperator{\cA}{\mathcal{A}}
\DeclareMathOperator{\cC}{\mathcal{C}}
\DeclareMathOperator{\cF}{\mathcal{F}}

\DeclareMathOperator{\g}{\gamma}
\DeclareMathOperator{\de}{\delta}

\title{Density of $g$-vector cones from triangulated surfaces}
\author{toshiya yurikusa}
\address{T. Yurikusa: Graduate School of Mathematics, Nagoya University, Chikusa-ku, Nagoya, 464-8602 Japan}
\email{m15049q@math.nagoya-u.ac.jp}

\begin{document}

\begin{abstract}
We study $g$-vector cones associated with clusters of cluster algebras defined from a marked surface $(S,M)$ of rank $n$. We determine the closure of the union of $g$-vector cones associated with all clusters. It is equal to $\mathbb{R}^n$ except for a closed surface with exactly one puncture, in which case it is equal to the half space of a certain explicit hyperplane in $\mathbb{R}^n$. Our main ingredients are laminations on $(S,M)$, their shear coordinates and their asymptotic behavior under Dehn twists. As an application, if $(S,M)$ is not a closed surface with exactly one puncture, the exchange graph of cluster tilting objects in the corresponding cluster category is connected. If $(S,M)$ is a closed surface with exactly one puncture, it has precisely two connected components.
\end{abstract}

\keywords{cluster algebra, marked surface, lamination, shear coordinate, cluster category, $\tau$-tilting theory}
\maketitle

\section{Introduction}

 Cluster algebras, introduced by Fomin and Zelevinsky in $2002$ \cite{FZ02}, are commutative algebras with generators called cluster variables. The certain tuples of cluster variables are called clusters. Their original motivation was to study total positivity of semisimple Lie groups and canonical bases of quantum groups.  In recent years, it has interacted with various subjects in mathematics, for example, representation theory of quivers, Poisson geometry, integrable systems, and so on.

 Let $Q$ be a quiver without loops and $2$-cycles, and let $\cA(Q)$ be the associated cluster algebra with principal coefficients (see Subsection \ref{clalg}). We denote by $\clus Q$ the set of clusters in $\cA(Q)$. Each cluster variable $x$ in $\cA(Q)$ has a numerical invariant $g_Q(x)$, called the $g$-vector of $x$ \cite{FZ07}. For each $\x \in \clus Q$, one can define a cone
\[
 C_Q(\x) := \biggl\{\sum_{x \in \x} a_x g_Q(x) \mid a_x \in \bR_{\ge0}\biggr\}
\]
in $\bR^n$, called the $g$-vector cone of $\x$. Note that these cones and their faces form a fan \cite[Theorem 8.7]{Re14a}. We say that $Q$ is finite type if $\clus Q < \infty$. The following result is well-known.

\begin{theorem}\cite[Theorem 10.6]{Re14a}\label{fineq}
 If a quiver $Q$ is finite type, then we have
\[
 \bigcup_{\x \in \clus Q}C_Q(\x)=\bR^n.
\]
\end{theorem}

 In this paper, we study an analogue of Theorem \ref{fineq} for cluster algebras defined from marked surfaces that were developed in \cite{FG06,FG09,FST,FoT,GSV}.

 Let $(S,M)$ be a marked surface and $T$ a tagged triangulation of $(S,M)$ (see Subsection \ref{tri}). We denote by $|T|$ the number of tagged arcs of $T$. Fomin, Shapiro and Thurston \cite{FST} constructed a quiver $Q_T$ associated with $T$. In $\cA(Q_T)$, cluster variables correspond to tagged arcs, and clusters correspond to tagged triangulations (Theorem \ref{bijtc}). Our first aim is to give the following analogue of Theorem \ref{fineq}.

\begin{theorem}\label{gfan}
 If $(S,M)$ is not a closed surface with exactly one puncture, then we have
\[
 \overline{\bigcup_{\x \in \clus Q_T}C_{Q_T}(\x)}=\bR^{|T|},
\]
where $\overline{(-)}$ is the closure with respect to the natural topology on $\bR^{|T|}$. If $(S,M)$ is a closed surface with exactly one puncture, then we have
\[
 \overline{\bigcup_{\x \in \clus Q_T}C_{Q_T}(\x)} = \overline{\bigcup_{\x \in \clus Q_T^{\rm op}}C_{Q_T^{\rm op}}(\x)} = \Bigl\{(a_{\de})_{\de \in T} \in \bR^{|T|} \mid \sum_{\de \in T} a_{\de} \ge 0\Bigr\}.
\]
\end{theorem}

 The second aim of this paper is to apply Theorem \ref{gfan} to representation theory. We consider a non-degenerate potential $W$ of $Q_T$ such that the associated Jacobian algebra $J(Q_T,W)$ is finite dimensional \cite{DWZ}. Such a potential $W$ exists (Proposition \ref{existW}). The potential $W$ and its Jacobian algebra $J=J(Q_T,W)$ have been studied by a number of researchers (see e.g. \cite{ABCP,CL,GLS,V}). We focus on the associated cluster category in this paper.

 Using the Ginzburg differential graded algebra $\Gamma=\Gamma_{Q_T,W}$ associated with $(Q_T,W)$ \cite{G}, Amiot \cite{A} constructed a generalized cluster category $\cC=\cC_{Q_T,W}$ with cluster tilting object $\Gamma$. The $g$-vector of each rigid object (resp., $\tau$-rigid pair) in $\cC$ (resp., $\Mod J$) is a certain element in the Grothendieck group $K_0(\add \Gamma)$ (resp., $K_0(\proj J)$). The $g$-vectors of indecomposable direct summands of a cluster tilting object $X$ (resp., a $\tau$-tilting pair $(M,P)$) form a cone $C_{\Gamma}(X)$ in $K_0(\add \Gamma)\otimes_{\bZ}\bR$ (resp., $C_J(M,P)$ in $K_0(\proj J)\otimes_{\bZ}\bR$), called the $g$-vector cone of $X$ (resp., $(M,P)$). Note that these $g$-vector cones and their faces form a fan \cite{DIJ}. Such a fan plays an important role in the study of scattering diagrams and their wall-chamber structures (see e.g. \cite{B,BST,GHKK,GS,KS,Y18a}).

 We denote by $\ctilt\cC$ (resp., $\sttilt J$) the set of isomorphism classes of basic cluster tilting objects in $\cC$ (resp., $\tau$-tilting pairs in $\Mod J$). We also denote by $\ctilt^+\cC$ (resp., $\ctilt^-\cC$, $\sttilt^+ J$, $\sttilt^- J$) the subset of $\ctilt\cC$ (resp., $\ctilt\cC$, $\sttilt J$, $\sttilt J$) consisting of mutation equivalence classes containing $\Gamma$ (resp., $\Gamma[1]$, $(J,0)$, $(0,J)$). We set
\[
 \ctilt^{\pm}\cC := \ctilt^+\cC \cup \ctilt^-\cC\ \mbox{and}\ \sttilt^{\pm} J := \sttilt^+ J \cup \sttilt^- J.
\]
 The following analogues of Theorem \ref{gfan} hold.

\begin{theorem}\label{gtame}
 Let $T$ be a tagged triangulation of a marked surface $(S,M)$. For a non-degenerate potential $W$ of $Q_T$ such that $J=J(Q_T,W)$ is finite dimensional, let $\cC=\cC_{Q_T,W}$ and $\Gamma=\Gamma_{Q_T,W}$. Then we have the equalities
\[
 \overline{\bigcup_{U\in\ctilt^{\pm}\cC}C_{\Gamma}(U)}=K_0(\add \Gamma)\otimes_{\bZ}\bR\ \ \text{and}\ \ \overline{\bigcup_{(M,P)\in\sttilt^{\pm} J}C_J(M,P)}=K_0(\proj J)\otimes_{\bZ}\bR.
\]
\end{theorem}

 This theorem means that $g$-vector cones are dense in the scattering diagram of $J$. It gives the following application.

\begin{corollary}\label{connect}
 Any basic cluster tilting object in $\cC$ (resp., $\tau$-tilting pair in $\Mod J$) is contained in $\ctilt^{\pm}\cC$ (resp., $\sttilt^{\pm} J$). In particular, if $(S,M)$ is not a closed surface with exactly one puncture, the exchange graph of $\ctilt\cC$ (resp., $\sttilt J$) is connected, thus $\ctilt\cC = \ctilt^+\cC = \ctilt^-\cC$ (resp., $\sttilt J = \sttilt^+ J  = \sttilt^- J$). Otherwise, it has precisely two connected components $\ctilt^+\cC$ and $\ctilt^-\cC$ (resp., $\sttilt^+ J$ and $\sttilt^- J$).
\end{corollary}

 Notice that it was known by Plamondon \cite{Pl13} and Ladkani \cite{Lad13} that if $(S,M)$ is a closed surface with exactly one puncture, then the exchange graph of $\ctilt\cC$ is not connected. Also, it was known by Qiu and Zhou \cite{QZ} that if $(S,M)$ has non-empty boundary, then the exchange graph of $\ctilt\cC$ is connected. Our proof is entirely different from theirs.

 To prove Theorem \ref{gfan}, our main ingredient is shear coordinates on $(S,M)$. To study coefficients in cluster algebras defined from $T$, Fomin and Thurston \cite{FoT} used a certain class of curves in $S$, called laminates, and finite multi-sets of pairwise non-intersecting laminates, called laminations (see also \cite{FG07,T}). To a laminate $\ell$ of $(S,M)$, they associated an integer vector $b_T(\ell) \in \bZ^{|T|}$ whose entries are shear coordinates of $\ell$ and defined $b_T(L) := \sum_{\ell \in L} b_T(\ell) \in \bZ^{|T|}$ for a lamination $L$ on $(S,M)$. They showed that the map $L \mapsto b_T(L)$ induces a bijection between the set of laminations on $(S,M)$ and $\bZ^{|T|}$. Using this bijection, some properties of cluster algebras were given (see e.g. \cite{MSW13,Re14b}). In this paper, we want to consider not only integer vectors but real vectors.

 For a multi-set $L$ of laminates of $(S,M)$, in the same way as $g$-vector cones, we can define a cone $C_T(L)$ in $\bR^{|T|}$, called the shear coordinate cone of $L$ with respect to $T$. Recall that there is a natural injective map $\e$ from the set of tagged arcs of $(S,M)$ to the set of laminates of $(S,M)$ (see Subsection \ref{lami}). We denote by $\bT$ the set of tagged triangulations of $(S,M)$. The following result plays an important role to prove Theorem \ref{gfan}.

\begin{theorem}\label{main}
 Let $T$ be a tagged triangulation of a marked surface $(S,M)$. Then we have
\[
 \overline{\bigcup_{T' \in \bT}C_T(\e(T'))}=\bR^{|T|}.
\]
 If $(S,M)$ is a closed surface with exactly one puncture $p$, then we have
\[
 \overline{\bigcup_{T' \in \bT^+}C_T(\e(T'))} = \overline{\bigcup_{T' \in \bT^-}(-C_T(\e(T')))} = \Bigl\{(a_{\de})_{\de \in T} \in \bR^{|T|} \mid \sum_{\de \in T} a_{\de} \le 0\Bigr\},
\]
where $\bT^+$ (resp., $\bT^-$) is the set of tagged triangulations of $(S,M)$ tagged at $p$ in the same (resp., different) way as $T$.
\end{theorem}

 It will be interesting to understand connections between our results and known results on Teichm\"{u}ller spaces such as \cite{FG11,Ro12,Ro13}.

 This paper is organized as follows. In Section \ref{pfmain}, we recall the notions of marked surfaces, laminations, and their shear coordinates. We study shear coordinates of laminates and their asymptotic behavior under Dehn twists, and prove Theorem \ref{main}. In Section \ref{pfgfan}, we recall cluster algebras defined from triangulated surfaces. We show that the shear coordinate of a laminate with respect to $T$ correspond with the $g$-vector of a cluster variable in $\cA(Q)$ or $\cA(Q^{\rm op})$. Consequently, Theorem \ref{gfan} follows from Theorem \ref{main}. In Section \ref{Rep}, we recall $\tau$-tilting theory, cluster tilting theory, and the relationships between them and cluster algebras. Finally, we prove Theorem \ref{gtame} and Corollary \ref{connect}.

\medskip\noindent{\bf Acknowledgements}.
 The author would like to thank his supervisor Osamu Iyama for his guidance and helpful comments. He also thanks Daniel Labardini-Fragoso for valuable comments and the referees for fruitful suggestions. He is a Research Fellow of Society for the Promotion of Science (JSPS). This work was supported by JSPS KAKENHI Grant Number JP17J04270.

\section{Density of shear coordinate cones from triangulated surfaces}\label{pfmain}

\subsection{Marked surfaces and tagged triangulations}\label{tri}

 We start with recalling the notions of \cite{FST}. Let $S$ be a connected compact oriented Riemann surface with (possibly empty) boundary $\partial S$ and $M$ a non-empty finite set of marked points on $S$ with at least one marked point on each boundary component. We call the pair $(S,M)$ a {\it marked surface}. Any marked point in the interior of $S$ is called a {\it puncture}. For technical reasons, we assume that $(S,M)$ is neither a monogon with at most one puncture, a digon without punctures, a triangle without punctures, nor a sphere with at most three punctures.

 An {\it arc} $\g$ of $(S,M)$ is a curve in $S$ with endpoints in $M$, considered up to isotopy, such that the following conditions are satisfied:
\begin{itemize}
 \item $\g$ does not intersect itself except at its endpoints;
 \item $\g$ is disjoint from $M$ and $\partial S$ except at its endpoints;
 \item $\g$ does not cut out an unpunctured monogon or an unpunctured digon.
\end{itemize}
 An arc with two identical endpoints is called a {\it loop}. 
 Two arcs are called {\it compatible} if they don't intersect in the interior of $S$. When we consider intersections of curves $\g$ and $\de$, we assume that $\g$ and $\de$ intersect transversally in a minimum number of points. We denote by $\g \cap \de$ the set of their intersection points. An {\it ideal triangulation} is a maximal collection of distinct pairwise compatible arcs.
 A triangle with only two distinct sides is called {\it self-folded} (see Figure \ref{self-folded}).
\begin{figure}[htp]
\begin{minipage}{0.5\textwidth}
\centering
\begin{tikzpicture}
 \coordinate (0) at (0,0);
 \coordinate (1) at (0,-1);
 \draw (0) node[left]{$p$} to node[right,pos=0.3]{$\g'$} (1) node[below]{$o$};
 \draw (0,0.4) node[above]{$\g$} ..controls(0.5,0.4)and(1,-0.2).. (1);
 \draw (0,0.4) ..controls(-0.5,0.4)and(-1,-0.2).. (1);
 \fill(0) circle (0.7mm); \fill (1) circle (0.7mm);
\end{tikzpicture}
\hspace{10mm}
\begin{tikzpicture}[baseline=-15mm]
 \coordinate (0) at (0,0);
 \coordinate (1) at (0,-1);
 \draw (0) node[left]{$p$} to node[left]{$\iota(\g')$} (1) node[below]{$o$};
 \draw (0) to [out=80,in=100,relative] node[pos=0.2]{\rotatebox{40}{\footnotesize $\bowtie$}} node[right]{$\iota(\g)$} (1);
 \fill(0) circle (0.7mm); \fill (1) circle (0.7mm);
\end{tikzpicture}
   \caption{A self-folded triangle and the corresponding tagged arcs}
   \label{self-folded}
\end{minipage}
\begin{minipage}{0.48\textwidth}\vspace{5mm}
\centering
\begin{tikzpicture}
 \coordinate (0) at (0,0);
 \coordinate (1) at (0,-1.2);
 \draw (0) to node[left]{$\de$} (1);
 \draw (0) to [out=80,in=100,relative] node[pos=0.2]{\rotatebox{40}{\footnotesize $\bowtie$}} node[right]{$\varepsilon$} (1);
 \fill(0) circle (0.7mm); \fill (1) circle (0.7mm);
\end{tikzpicture}
\hspace{10mm}
\begin{tikzpicture}
 \coordinate (0) at (0,0);
 \coordinate (1) at (0,-1.2);
 \draw (0) to node[left]{$\de$} node[pos=0.8]{\rotatebox{0}{\footnotesize $\bowtie$}} (1);
 \draw (0) to [out=80,in=100,relative] node[pos=0.2]{\rotatebox{40}{\footnotesize $\bowtie$}} node[pos=0.75]{\rotatebox{-40}{\footnotesize $\bowtie$}} node[right]{$\varepsilon$} (1);
 \fill(0) circle (0.7mm); \fill (1) circle (0.7mm);
\end{tikzpicture}
   \caption{Pairs of conjugate arcs $(\de,\varepsilon)$}
   \label{pair}
\end{minipage}
\end{figure}
 For an ideal triangulation $T$, a {\it flip} at an arc $\g \in T$ replaces $\g$ with another arc $\g' \notin T$ such that $(T\setminus\{\g\})\cup\{\g'\}$ is an ideal triangulation. Notice that an arc inside a self-folded triangle can not be flipped. To make flip always possible, the notion of tagged arcs was introduced in \cite{FST}.

 A {\it tagged arc} $\de$ of $(S,M)$ is an arc whose each end is tagged in one of two ways, {\it plain} or {\it notched}, such that the following conditions are satisfied:
\begin{itemize}
 \item $\de$ does not cut out a monogon with exactly one puncture;
 \item If an endpoint of $\de$ lie on $\partial S$, then it is tagged plain;
 \item If $\de$ is a loop, then the both ends are tagged in the same way.
\end{itemize}
 In the figures, we represent tagged arcs as follows:
\[
\begin{tikzpicture}
 \coordinate (0) at (0,0) node[left]{plain};
 \coordinate (1) at (1,0);   \fill (1) circle (0.7mm);
 \draw (0) to (1);
\end{tikzpicture}
\hspace{7mm}
\begin{tikzpicture}
 \coordinate (0) at (0,0) node[left]{notched};
 \coordinate (1) at (1,0);   \fill (1) circle (0.7mm);
 \draw (0) to node[pos=0.8]{\rotatebox{90}{\footnotesize $\bowtie$}} (1);
\end{tikzpicture}
\]

 For an arc $\g$ of $(S,M)$, we define a tagged arc $\iota(\g)$ as follows:
\begin{itemize}
 \item If $\g$ does not cut out a monogon with exactly one puncture, then $\iota(\g)$ is the tagged arc obtained from $\g$ by tagging both ends plain;
 \item If $\g$ is a loop at $o \in M$ cutting out a monogon with exactly one puncture $p$, then there is a unique arc $\g'$ that connects $o$ and $p$ and does not intersect $\g$. And then $\iota(\g)$ is the tagged arc obtained by tagging $\g'$ plain at $o$ and notched at $p$ (see Figure \ref{self-folded}).
\end{itemize}
 A {\it pair of conjugate arcs} is, for a self-folded triangle $\{\g,\g'\}$, $(\iota(\g),\iota(\g'))$ or a pair obtained from $(\iota(\g),\iota(\g'))$ by simultaneous changing tags at each endpoint (see Figure \ref{pair}).

 For a tagged arc $\de$, we denote by $\de^{\circ}$ the arc obtained from $\de$ by forgetting its tags. Two tagged arcs $\de$ and $\epsilon$ are called {\it compatible} if the following conditions are satisfied:
\begin{itemize}
 \item The arcs $\de^{\circ}$ and $\epsilon^{\circ}$ are compatible;
 \item If $\de^{\circ}=\epsilon^{\circ}$, then at least one end of $\epsilon$ is tagged in the same way as the corresponding end of $\de$;
 \item If $\de^{\circ}\neq\epsilon^{\circ}$ and they have a common endpoint $o$, then the ends of $\de$ and $\epsilon$ at $o$ are tagged in the same way.
\end{itemize}
 A {\it partial tagged triangulation} is a collection of distinct pairwise compatible tagged arcs. If a partial tagged triangulation is maximal, then it is called a {\it tagged triangulation}. Recall that we denote by $\bT$ the set of tagged triangulations of $(S,M)$. We can define {\it flips} of tagged triangulations in the same way as ones of ideal triangulations. In particular, any tagged arc can be flipped.

\begin{theorem} \cite[Theorem 7.9, Proposition 7.10]{FST}\label{flip}
 If $(S,M)$ is not a closed surface with exactly one puncture, the exchange graph of $\bT$ is connected, that is, any two tagged triangulations of $(S,M)$ are connected by a sequence of flips. Otherwise, it has exactly two isomorphic components: one in which all ends of tagged arcs are plain and one in which they are notched.
\end{theorem}

\subsection{Laminations on marked surfaces}\label{lami}

 We recall the notions of \cite{FoT}. A {\it laminate} of $(S,M)$ is a non-self-intersecting curve in $S$, considered up to isotopy relative to $M$, which is either
\begin{itemize}
 \item a closed curve, or
 \item a curve whose ends are unmarked points on $\partial S$ or spirals around punctures (either clockwise or counterclockwise),
\end{itemize}
 and the following curves are not allowed (see Figure \ref{fignonlam}):
\begin{itemize}
 \item a curve cutting out a disk with at most one puncture;
 \item a curve with two endpoints on $\partial S$ such that it is isotopic to a piece of $\partial S$ containing at most one marked point;
 \item a curve whose both ends are spirals around a common puncture in the same direction such that it does not enclose anything else.
\end{itemize}

\begin{figure}[htp]
\begin{minipage}{0.48\textwidth}
\centering
\begin{tikzpicture}[baseline=0mm]
 \coordinate (0) at (0,0);
 \coordinate (l) at (-1,0.6);
 \coordinate (r) at (1,0.6);
 \draw (2,2)--(-2,2)--(-2,-2)--(2,-2)--(2,2); \draw[pattern=north east lines] (0,-0.8) circle (4mm);
 \draw[blue] (-2,1) arc (-90:0:10mm); \draw[blue] (l) circle (4mm); \draw[blue] (-1.2,-0.7) circle (5mm);
 \draw[blue] (1,1.1) arc (90:-90:4mm); \draw[blue] (1,0.3) arc (-90:-280:2.4mm);
 \draw[blue] (0.6,0.6) .. controls (0.6,1) and (1.2,1) .. (1.2,0.6); \draw[blue] (1.2,0.6) arc (0:-120:1.8mm);
 \draw[blue] (0.6,0.6) arc (-0:-270:5mm); \draw[blue] (0.1,1.1) .. controls (0.4,1.13) and (0.8,1.13) .. (1,1.1);
 \draw[blue] (2,0) arc (90:270:7mm);
 \fill (l) circle (0.7mm); \fill (r) circle (0.7mm); \fill (2,2) circle (0.7mm); \fill (-2,2) circle (0.7mm); \fill (-2,-2) circle (0.7mm); \fill (2,-2) circle (0.7mm); \fill (0,-0.4) circle (0.7mm);
\end{tikzpicture}
   \caption{Curves which are not laminates}
   \label{fignonlam}
\end{minipage}
\begin{minipage}{0.48\textwidth}
\[
\begin{tikzpicture}[baseline=0mm]
 \coordinate (0) at (0,0);
 \coordinate (l) at (-1,0.6);
 \coordinate (r) at (1,0.6);
 \draw (2,2)--(-2,2)--(-2,-2)--(2,-2)--(2,2); \draw[pattern=north east lines] (0,-0.8) circle (4mm);
 \draw[blue] (-2,1.5)--(2,1.5); \draw[blue] (0,-0.8) circle (7mm); \draw[blue] (0,-0.8) circle (5.5mm);
 \draw[blue] (-1,1) .. controls (-0.3,0.95) and (0.3,0.95) .. (1,1);
 \draw[blue] (-1,1) arc (90:270:3.5mm); \draw[blue] (-1,0.3) arc (-90:100:2.4mm);
 \draw[blue] (1,1) arc (90:-90:3.5mm); \draw[blue] (1,0.3) arc (-90:-280:2.4mm);
 \draw[blue] (2,0) .. controls (0.7,-0.2) and (0.6,0.3) .. (0.6,0.6);
 \draw[blue] (0.6,0.6) .. controls (0.6,1) and (1.2,1) .. (1.2,0.6);
 \draw[blue] (1.2,0.6) arc (0:-120:1.8mm);
 \draw[blue] (-1,-2) .. controls (-1,0) and (-0.5,0.2) .. (0,0.2);
 \draw[blue] (1,-2) .. controls (1,0) and (0.5,0.2) .. (0,0.2);
 \fill (l) circle (0.7mm); \fill (r) circle (0.7mm); \fill (2,2) circle (0.7mm); \fill (-2,2) circle (0.7mm); \fill (-2,-2) circle (0.7mm); \fill (2,-2) circle (0.7mm); \fill (0,-0.4) circle (0.7mm);
\end{tikzpicture}
\]
   \caption{A lamination on an annulus with $2$ punctures}
   \label{figlam}
\end{minipage}
\end{figure}

\begin{definition} 
 We say that two laminates of $(S,M)$ are {\it compatible} if they don't intersect. A finite multi-set of pairwise compatible laminates of $(S,M)$ is called a {\it lamination} on $(S,M)$ (see Figure \ref{figlam}).
\end{definition}

 Let $\ell$ be a laminate of $(S,M)$. For an ideal/tagged triangulation $T$ of $(S,M)$, we define the {\it shear coordinate $b_{\g,T}(\ell)$ of $\ell$} with respect to $\g \in T$ (see \cite[Definition 12.2, 13.1]{FoT}).

 First, we assume that $T$ is an ideal triangulation. If $\g \in T$ is not inside a self-folded triangle of $T$, then $b_{\g,T}(\ell)$ is defined by a sum of contributions from all intersections of $\g$ and $\ell$ as follows: Such an intersection contributes $+1$ (resp., $-1$) to $b_{\g,T}(\ell)$ if a segment of $\ell$ cuts through the quadrilateral surrounding $\g$ as in the left (resp., right) diagram of Figure \ref{SZ}.
\begin{figure}[htp]
\centering
$+1$
\begin{tikzpicture}[baseline=0mm]
 \coordinate (u) at (0,2);
 \coordinate (l) at (-1,1);
 \coordinate (r) at (1,1);
 \coordinate (d) at (0,0);
 \coordinate (s) at (-1,0.5);
 \coordinate (t) at (1,1.5);
 \draw (u)--(l)--(d)--(r)--(u)--node[fill=white,inner sep=2,pos=0.4]{$\g$}(d);
 \draw (s)..controls (-0.1,0.7) and (0.3,1)..(t)node[above]{$\ell$};
\end{tikzpicture}
   \hspace{20mm}
\begin{tikzpicture}[baseline=0mm]
 \coordinate (u) at (0,2);
 \coordinate (l) at (-1,1);
 \coordinate (r) at (1,1);
 \coordinate (d) at (0,0);
 \coordinate (s) at (-1,1.5);
 \coordinate (t) at (1,0.5);
 \draw (u)--(l)--(d)--(r)--(u)--node[fill=white,inner sep=2,pos=0.4]{$\g$}(d);
 \draw (s)node[above]{$\ell$}..controls (-0.3,0.9) and (0.3,0.6)..(t);
\end{tikzpicture}
 $-1$
   \caption{The contribution from a segment of the laminate $\ell$ on the left (resp., right) is $+1$ (resp., $-1$)}
   \label{SZ}
\end{figure}
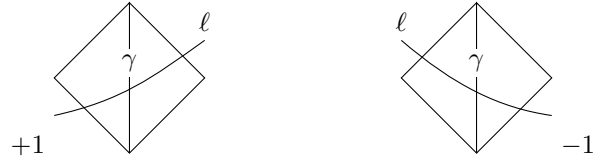
 Suppose that $\g \in T$ is inside a self-folded triangle $\{\g,\g'\}$ of $T$, where $\g'$ is a loop enclosing exactly one puncture $p$. Then we define $b_{\g,T}(\ell) = b_{\g',T}(\ell^{(p)})$, where $\ell^{(p)}$ is a laminate obtained from $\ell$ by changing the directions of its spirals at $p$ if they exist.

 Next, we assume that $T$ is a tagged triangulation. If there is an ideal triangulation $T^0$ satisfying $T=\iota(T^0)$, then we define $b_{\g,T}(\ell) = b_{\g^0,T^0}(\ell)$, where $\g=\iota(\g^0)$. For an arbitrary $T$, we can obtain a tagged triangulation $T^{(p_1\cdots p_m)}$ from $T$ by simultaneous changing all tags at punctures $p_1,\ldots,p_m$ (possibly $m=0$), in such a way that there is a unique ideal triangulation $T^0$ satisfying $T^{(p_1\cdots p_m)}=\iota(T^0)$ (see \cite[Remark 3.11]{MSW11}). Then we define $b_{\g,T}(\ell) = b_{\g^{(p_1\cdots p_m)},T^{(p_1\cdots p_m)}}\bigl((\cdots((\ell^{(p_1)})^{(p_2)})\cdots)^{(p_m)}\bigr)$, where $\g^{(p_1\cdots p_m)}$ corresponds to $\g$.

 For a multi-set $L=L' \sqcup \{\ell\}$ of laminates of $(S,M)$, the {\it shear coordinate $b_{\g,T}(L)$ of $L$} with respect to $\g \in T$ is inductively defined by
\[
 b_{\g,T}(L) = b_{\g,T}(L') + b_{\g,T}(\ell).
\]
 We denote by $b_T(L)$ a vector $(b_{\g,T}(L))_{\g \in T} \in \bZ^{|T|}$. Note that the shear coordinate cone $C_T(L)$ is a cone spanned by $b_T(\ell)$ for $\ell \in L$. These vectors have the following property.

\begin{theorem}\cite[Theorems 12.3, 13.6]{FoT}\label{thurston}
 Let $T$ be a tagged triangulation of $(S,M)$. The map sending laminations $L$ to $b_T(L)$ induces a bijection between the set of laminations on $(S,M)$ and $\bZ^{|T|}$.
\end{theorem}

\begin{example}\label{ex1}
 For a digon $(S,M)$ with exactly one puncture, all laminates are given as follows:
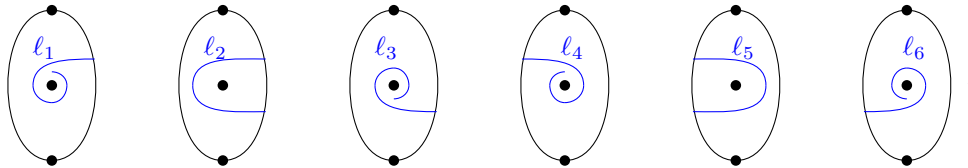
\[
\begin{tikzpicture}[baseline=-10mm]
 \coordinate (0) at (0,0);
 \coordinate (1) at (0,-1);
 \coordinate (2) at (0,-2);
 \coordinate (3) at (0.57,-0.65);
 \draw (2) to [out=180,in=180] (0); \draw (2) to [out=0,in=0] (0);
 \draw[blue] (3) to [out=180,in=90] (-0.25,-1);
 \draw[blue] (-0.25,-1) to [out=-90,in=180] (0,-1.23);
 \draw[blue] (0,-1.23) to [out=0,in=-90] (0.2,-1);
 \draw[blue] (0.2,-1) to [out=90,in=0] (0,-0.82);
 \fill(0) circle (0.7mm); \fill (1) circle (0.7mm); \fill (2) circle (0.7mm);
 \node[blue] at(-0.1,-0.5) {$\ell_1$};
\end{tikzpicture}
   \hspace{7mm}
\begin{tikzpicture}[baseline=-10mm]
 \coordinate (0) at (0,0);
 \coordinate (1) at (0,-1);
 \coordinate (2) at (0,-2);
 \coordinate (3) at (0.57,-0.65);
 \coordinate (5) at (0.57,-1.35);
 \draw (2) to [out=180,in=180] (0); \draw (2) to [out=0,in=0] (0);
 \draw[blue] (3) to [out=180,in=90] (-0.4,-1);
 \draw[blue] (-0.4,-1) to [out=-90,in=180] (5);
 \fill(0) circle (0.7mm); \fill (1) circle (0.7mm); \fill (2) circle (0.7mm);
 \node[blue] at(-0.1,-0.5) {$\ell_2$};
\end{tikzpicture}
   \hspace{7mm}
\begin{tikzpicture}[baseline=-10mm]
 \coordinate (0) at (0,0);
 \coordinate (1) at (0,-1);
 \coordinate (2) at (0,-2);
 \coordinate (3) at (0.57,-1.35);
 \draw (2) to [out=180,in=180] (0); \draw (2) to [out=0,in=0] (0);
 \draw[blue] (3) to [out=180,in=-90] (-0.25,-1);
 \draw[blue] (-0.25,-1) to [out=90,in=180] (0,-0.77);
 \draw[blue] (0,-0.77) to [out=0,in=90] (0.2,-1);
 \draw[blue] (0.2,-1) to [out=-90,in=0] (0,-1.18);
 \fill(0) circle (0.7mm); \fill (1) circle (0.7mm); \fill (2) circle (0.7mm);
 \node[blue] at(-0.1,-0.5) {$\ell_3$};
\end{tikzpicture}
   \hspace{7mm}
\begin{tikzpicture}[baseline=-10mm]
 \coordinate (0) at (0,0);
 \coordinate (1) at (0,-1);
 \coordinate (2) at (0,-2);
 \coordinate (3) at (-0.57,-0.65);
 \draw (2) to [out=180,in=180] (0); \draw (2) to [out=0,in=0] (0);
 \draw[blue] (3) to [out=0,in=90] (0.25,-1);
 \draw[blue] (0.25,-1) to [out=-90,in=0] (0,-1.23);
 \draw[blue] (0,-1.23) to [out=180,in=-90] (-0.2,-1);
 \draw[blue] (-0.2,-1) to [out=90,in=180] (0,-0.82);
 \fill(0) circle (0.7mm); \fill (1) circle (0.7mm); \fill (2) circle (0.7mm);
 \node[blue] at(0.1,-0.5) {$\ell_4$};
\end{tikzpicture}
   \hspace{7mm}
\begin{tikzpicture}[baseline=-10mm]
 \coordinate (0) at (0,0);
 \coordinate (1) at (0,-1);
 \coordinate (2) at (0,-2);
 \coordinate (3) at (-0.57,-0.65);
 \coordinate (5) at (-0.57,-1.35);
 \draw (2) to [out=180,in=180] (0); \draw (2) to [out=0,in=0] (0);
 \draw[blue] (3) to [out=0,in=90] (0.4,-1);
 \draw[blue] (0.4,-1) to [out=-90,in=0] (5);
 \fill(0) circle (0.7mm); \fill (1) circle (0.7mm); \fill (2) circle (0.7mm);
 \node[blue] at(0.1,-0.5) {$\ell_5$};
\end{tikzpicture}
   \hspace{7mm}
\begin{tikzpicture}[baseline=-10mm]
 \coordinate (0) at (0,0);
 \coordinate (1) at (0,-1);
 \coordinate (2) at (0,-2);
 \coordinate (3) at (-0.57,-1.35);
 \draw (2) to [out=180,in=180] (0); \draw (2) to [out=0,in=0] (0);
 \draw[blue] (3) to [out=0,in=-90] (0.25,-1);
 \draw[blue] (0.25,-1) to [out=90,in=0] (0,-0.77);
 \draw[blue] (0,-0.77) to [out=180,in=90] (-0.2,-1);
 \draw[blue] (-0.2,-1) to [out=-90,in=180] (0,-1.18);
 \fill(0) circle (0.7mm); \fill (1) circle (0.7mm); \fill (2) circle (0.7mm);
 \node[blue] at(0.1,-0.5) {$\ell_6$};
\end{tikzpicture}.
\]
 We consider the following tagged triangulation $T$:
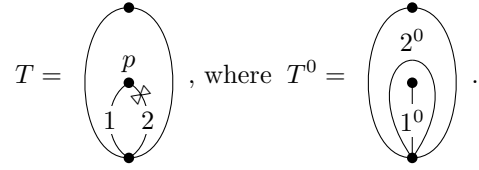
\[
 T=
\begin{tikzpicture}[baseline=-10mm]
 \coordinate (0) at (0,0);
 \coordinate (1) at (0,-1);
 \coordinate (2) at (0,-2);
 \draw (2) to [out=180,in=180] (0); \draw (2) to [out=0,in=0] (0);
 \draw (1) to [out=-60,in=-120,relative] node[fill=white,inner sep=2]{$1$} (2);
 \draw (1) to [out=60,in=120,relative] node[pos=0.2]{\rotatebox{40}{\footnotesize $\bowtie$}}node[fill=white,inner sep=2]{$2$} (2);
 \fill(0) circle (0.7mm); \fill (1) circle (0.7mm); \fill (2) circle (0.7mm);
 \node at(0,-0.75) {$p$};
\end{tikzpicture}
\text{, where}\hspace{2mm}
 T^0=
\begin{tikzpicture}[baseline=-10mm]
 \coordinate (0) at (0,0);
 \coordinate (1) at (0,-1);
 \coordinate (2) at (0,-2);
 \draw (2) to [out=180,in=180] (0); \draw (2) to [out=0,in=0] (0);
 \draw (1) to node[fill=white,inner sep=2]{$1^0$} (2);
 \fill(0) circle (0.7mm); \fill (1) circle (0.7mm); \fill (2) circle (0.7mm);
 \draw (2) .. controls (0.5,-1.3) and (0.3,-0.7) .. (0,-0.7);
 \draw (2) .. controls (-0.5,-1.3) and (-0.3,-0.7) .. (0,-0.7) node[above]{$2^0$};
\end{tikzpicture}.
\]
 The shear coordinate $b_{2,T}(\ell_1)$ is given by $b_{2^0,T^0}(\ell_1)=-1$. Since $\ell_3^{(p)}=\ell_1$, we have the equalities
\[
 b_{1,T}(\ell_3) = b_{1^0,T^0}(\ell_3) = b_{2^0,T^0}(\ell_3^{(p)}) = b_{2^0,T^0}(\ell_1) = -1.
\]
 Similarly, for $i \in \{1,2\}$ and $j \in \{1,\ldots,6\}$, the shear coordinates $b_{i,T}(\ell_j)$ and $b_{T}(\ell_j)$ are given as follows:
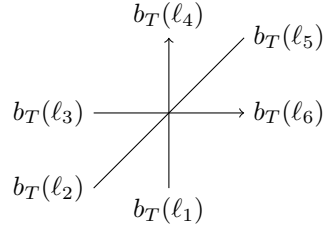
\[
\renewcommand{\arraystretch}{1.5}{
\begin{tabular}{c||c|c|c|c|c|c}
 i\ \textbackslash\ j & 1 & 2 & 3 & 4 & 5 & 6 \\\hline\hline
 1 & 0 & -1 & -1 & 0 & 1 & 1 \\\hline
 2 & -1 & -1 & 0 & 1 & 1 & 0
\end{tabular}}
   \hspace{10mm}
\begin{tikzpicture}[baseline=0mm,scale=1]
 \coordinate (0) at (0,0); \coordinate (x) at (1,0); \coordinate (-x) at (-1,0);
 \coordinate (y) at (0,1); \coordinate (-y) at (0,-1);
 \draw[->] (0)--(x) node[right]{$b_T(\ell_6)$};
 \draw (0)--(-x) node[left]{$b_T(\ell_3)$};
 \draw[->] (0)--(y) node[above]{$b_T(\ell_4)$};
 \draw (0)--(-y) node[below]{$b_T(\ell_1)$};
 \draw (0)--(1,1) node[right]{$b_T(\ell_5)$};
 \draw (0)--(-1,-1) node[left]{$b_T(\ell_2)$};
\end{tikzpicture}
\]
 In particular, we have
\[
 \bigcup_{j=1}^6 C_T(\{\ell_j,\ell_{j+1}\}) = \bR^2,
\]
 where $\ell_7 = \ell_1$. On the other hand, all laminations on $(S,M)$ are given by $\{m \ell_j,n \ell_{j+1}\}$ for $j \in \{1,\ldots,6\}$ and $m, n \in \bZ_{\ge 0}$. Since $C_T(\{\ell_{j-1},\ell_j\}) \cap C_T(\{\ell_j,\ell_{j+1}\})=C_T(\{\ell_j\})$ and $b_T$ induces a bijection
\[
 b_T : \{\{m \ell_j,n \ell_{j+1}\} \mid m, n \in \bZ_{\ge 0}\} \longleftrightarrow C_T(\{\ell_j,\ell_{j+1}\}) \cap  \bZ^2,
\]
 there is a bijection between the set of laminations on $(S,M)$ and $\bZ^2$.
\end{example}

\subsection{Elementary and exceptional laminates}

 Non-closed laminates of $(S,M)$ are divided into two types, elementary and exceptional. For a tagged arc $\de$ of $(S,M)$, we define an {\it elementary laminate} $\e(\de)$ as follows:
\begin{itemize}
 \item $\e(\de)$ is a laminate running along $\de$ in a small neighborhood of it;
 \item If $\de$ has an endpoint $o$ on a component $C$ of $\partial S$, then the corresponding endpoint of $\e(\de)$ is located near $o$ on $C$ in the clockwise direction as in the left diagram of Figure \ref{elelam};
 \item If $\de$ has an endpoint at a puncture $p$, then the corresponding end of $\e(\de)$ is a spiral around $p$ clockwise (resp., counterclockwise) if $\de$ is tagged plain (resp., notched) at $p$ as in the right diagram of Figure \ref{elelam}.
\end{itemize}
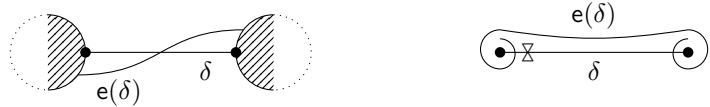
\begin{figure}[htp]
\centering
\begin{tikzpicture}[baseline=0mm]
 \coordinate (u) at (0,0.5);
 \coordinate (l) at (-0.5,0);
 \coordinate (r) at (0.5,0);
 \coordinate (d) at (0,-0.5);
 \coordinate (u1) at (3,0.5);
 \coordinate (l1) at (2.5,0);
 \coordinate (r1) at (3.5,0);
 \coordinate (d1) at (3,-0.5);
 \draw (r)--node[below,pos=0.8]{$\de$}(l1);
 \draw (0.4,-0.3) .. controls (1.5,-0.3) and (1.5,0.3) ..node[below,pos=0.2]{$\e(\de)$} (2.6,0.3);
 \draw[pattern=north east lines] (u) arc (90:-90:5mm);
 \draw[pattern=north east lines] (u1) arc (90:270:5mm);
 \draw[dotted] (u) arc (90:270:5mm);
 \draw[dotted] (u1) arc (90:-90:5mm);
 \fill (r) circle (0.7mm); \fill (l1) circle (0.7mm);
\end{tikzpicture}
   \hspace{20mm}
\begin{tikzpicture}[baseline=0mm]
 \coordinate (l) at (0,0);
 \coordinate (r) at (2.5,0);
 \draw (l)--node[below]{$\de$}node[pos=0.15]{\rotatebox{90}{\footnotesize $\bowtie$}}(r);
 \draw (0,0.3) .. controls (1,0.15) and (1.5,0.15) ..node[above]{$\e(\de)$} (2.5,0.3);
 \draw (0,0.3) arc (90:270:2.6mm);
 \draw (0,-0.22) arc (-90:90:2mm);
 \draw (2.5,0.3) arc (90:-90:2.6mm);
 \draw (2.5,-0.22) arc (-90:-270:2mm);
 \fill (l) circle (0.7mm); \fill (r) circle (0.7mm);
\end{tikzpicture}
   \caption{Elementary laminates of tagged arcs}
   \label{elelam}
\end{figure}
 It follows from the construction that the map $\e$ from the set of tagged arcs of $(S,M)$ to the set of laminates is injective. For an elementary laminate $\ell$, we denote by $\e^{-1}(\ell)$ a unique tagged arc $\de$ such that $\e(\de)=\ell$. Note that, for a tagged arc $\de$, a lamination $\{\e(\de)\}$ is a reflection of the elementary lamination of $\de$ defined in \cite[Definition 17.2]{FoT}. Our convention is more convenient for our aim.

 Elementary laminates have the following properties.

\begin{proposition}\label{bij}
 (1) Let $\de$ and $\de'$ be tagged arcs such that $\de^{\circ} \neq \de'^{\circ}$. Then $\de$ and $\de'$ are compatible if and only if $\e(\de)$ and $\e(\de')$ are compatible.\par
 (2) The map $\e$ induces a bijection between the set of partial tagged triangulations of $(S,M)$ without pairs of conjugate arcs and the set of laminations of $(S,M)$ consisting only of distinct elementary laminates.
\end{proposition}

\begin{proof}
 (1) Since $\e$ transforms $\de$ and $\de'$ just around the marked points of $(S,M)$, it is enough to consider neighborhoods of their endpoints. In particular, if $\de$ and $\de'$ have no common endpoints, the assertion holds. Suppose that $\de$ and $\de'$ have at least one common endpoint. Since $\de^{\circ} \neq \de'^{\circ}$, $(\de,\de')$ is not a pair of conjugate arcs. Thus $\de$ and $\de'$ are compatible if and only if the ends of $\de$ and $\de'$ at each common endpoint are tagged in the same way. By the definition of $\e$, it is equivalent that $\e(\de)$ and $\e(\de')$ are compatible.

 (2) If two distinct tagged arcs $\de$ and $\de'$ satisfying $\de^{\circ} = \de'^{\circ}$ are compatible, then $(\de,\de')$ is a pair of conjugate arcs, in which case $\e(\de)$ and $\e(\de')$ are not compatible. Therefore, the assertion follows from (1).
\end{proof}

 Laminates which are neither closed nor elementary are called {\it exceptional}. They are characterized as follows.

\begin{proposition}\label{excep}
 A laminate is exceptional if and only if it is one of the following curves (Figure \ref{exc}):
\begin{itemize}
 \item a curve enclosing exactly one puncture whose both endpoints lie on a common boundary segment;
 \item a curve enclosing exactly one puncture whose both ends are spirals around a common puncture in the same direction.
\end{itemize}
\end{proposition}

\begin{proof}
 Applying the same transformation as $\e^{-1}$ to a non-closed laminate $\ell$ and forgetting its tags, we obtain a unique ideal arc. In general, an ideal arc $\g$ is not obtained from a tagged arc by forgetting its tags if and only if $\g$ is a loop cutting out a monogon with exactly one puncture. Therefore, $\ell$ is not elementary if and only if it is one of the desired cases.
\end{proof}

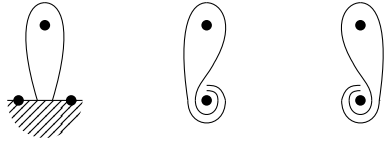
\begin{figure}[htp]
\centering
\begin{tikzpicture}[baseline=0mm]
 \coordinate (0) at (0,0);
 \coordinate (p) at (0,1);
 \draw (-0.1,0) .. controls (-0.4,1) and (-0.2,1.3) .. (0,1.3);
 \draw (0.1,0) .. controls (0.4,1) and (0.2,1.3) .. (0,1.3);
 \fill (p) circle (0.7mm);
 \draw (-0.5,0)--(0.5,0);
 \fill (-0.35,0) circle (0.7mm);   \fill (0.35,0) circle (0.7mm);
 \fill[pattern=north east lines] (-0.5,0) arc (-180:0:5mm);
\end{tikzpicture}
 \hspace{10mm}
\begin{tikzpicture}[baseline=0mm]
 \coordinate (0) at (0,0);
 \coordinate (p) at (0,1);
 \draw (-0.25,0) .. controls (-0.4,1) and (-0.2,1.3) .. (0,1.3);
 \draw (0.1,0.5) .. controls (0.4,1) and (0.2,1.3) .. (0,1.3);
 \fill (p) circle (0.7mm);
 \draw (-0.25,0) .. controls (-0.25,-0.4) and (0.25,-0.4) .. (0.25,0);
 \draw (0.25,0) .. controls (0.2,0.2) and (0.1,0.2) .. (0,0.2);
 \draw (0.1,0.5) .. controls (0,0.3) and (-0.15,0.2) .. (-0.15,0);
 \draw (-0.15,0) .. controls (-0.15,-0.25) and (0.15,-0.25) .. (0.15,0);
 \draw (0.15,0) .. controls (0.15,0.1) and (0.1,0.13) .. (0,0.13);
 \fill (0,0) circle (0.7mm);
\end{tikzpicture}
 \hspace{10mm}
\begin{tikzpicture}[baseline=0mm]
 \coordinate (0) at (0,0);
 \coordinate (p) at (0,1);
 \draw (-0.1,0.5) .. controls (-0.4,1) and (-0.2,1.3) .. (0,1.3);
 \draw (0.25,0) .. controls (0.4,1) and (0.2,1.3) .. (0,1.3);
 \fill (p) circle (0.7mm);
 \draw (0.25,0)--(0.25,0);
 \draw (0.25,0) .. controls (0.25,-0.4) and (-0.25,-0.4) .. (-0.25,0);
 \draw (-0.25,0) .. controls (-0.2,0.2) and (-0.1,0.2) .. (0,0.2);
 \draw (-0.1,0.5) .. controls (0,0.3) and (0.15,0.2) .. (0.15,0);
 \draw (0.15,0) .. controls (0.15,-0.25) and (-0.15,-0.25) .. (-0.15,0);
 \draw (-0.15,0) .. controls (-0.15,0.1) and (-0.1,0.13) .. (0,0.13);
 \fill (0,0) circle (0.7mm);
\end{tikzpicture}
   \caption{Exceptional laminates}
   \label{exc}
\end{figure}

 Note that exceptional laminates coincide with excluded curves for quasi-laminations in \cite{Re14b}. To interpret shear coordinates of exceptional laminates as ones of elementary laminates, we introduce the following notations. For an exceptional laminate $\ell$ of $(S,M)$, elementary laminates $\ell_{\sf p}$ and $\ell_{\sf q}$ are given by
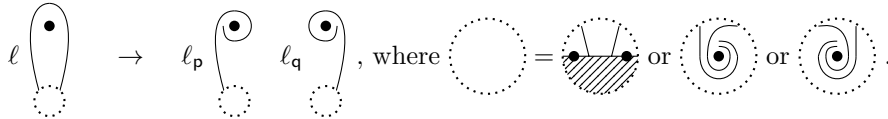
\[
 \ell
\begin{tikzpicture}[baseline=5mm]
 \coordinate (0) at (0,0);
 \coordinate (p) at (0,1);
 \draw (-0.1,0) .. controls (-0.4,1) and (-0.2,1.3) .. (0,1.3);
 \draw (0.1,0) .. controls (0.4,1) and (0.2,1.3) .. (0,1.3);
 \fill (p) circle (0.7mm);   \filldraw[dotted,thick,fill=white] (0) circle (2mm);
\end{tikzpicture}
 \hspace{4mm}\rightarrow\hspace{4mm}
 \ell_{\sf p}
\begin{tikzpicture}[baseline=5mm]
 \coordinate (0) at (0,0);
 \coordinate (p) at (0,1);
 \draw (-0.1,0) .. controls (-0.4,1) and (-0.2,1.2) .. (0,1.2);
 \draw (0,1.2) .. controls (0.3,1.2) and (0.3,0.8) .. (0,0.8);
 \draw (0,0.8) .. controls (-0.13,0.8) and (-0.14,0.9) .. (-0.15,1);
 \fill (p) circle (0.7mm);   \filldraw[dotted,thick,fill=white] (0) circle (2mm);
\end{tikzpicture}
 \hspace{3mm}
 \ell_{\sf q}
\begin{tikzpicture}[baseline=5mm]
 \coordinate (0) at (0,0);
 \coordinate (p) at (0,1);
 \draw (0.1,0) .. controls (0.4,1) and (0.2,1.2) .. (0,1.2);
 \draw (0,1.2) .. controls (-0.3,1.2) and (-0.3,0.8) .. (0,0.8);
 \draw (0,0.8) .. controls (0.13,0.8) and (0.14,0.9) .. (0.15,1);
 \fill (p) circle (0.7mm);   \filldraw[dotted,thick,fill=white] (0) circle (2mm);
\end{tikzpicture}
 \text{, where}\ 
\begin{tikzpicture}[baseline=-1mm]
 \draw[dotted,thick] (0,0) circle (5mm);
\end{tikzpicture}
 =
\begin{tikzpicture}[baseline=-1mm]
 \draw[dotted,thick] (0,0) circle (5mm);
 \draw (-0.15,0)--(-0.25,0.4) (0.15,0)--(0.25,0.4) (-0.5,0) to (0.5,0);
 \fill (-0.35,0) circle (0.7mm);   \fill (0.35,0) circle (0.7mm);
 \fill[pattern=north east lines] (-0.5,0) arc (-180:0:5mm);
\end{tikzpicture}
\ \text{or}\ 
\begin{tikzpicture}[baseline=-1mm]
 \draw[dotted,thick] (0,0) circle (5mm);
 \draw (-0.25,0.4)--(-0.25,0);
 \draw (-0.25,0) .. controls (-0.25,-0.4) and (0.25,-0.4) .. (0.25,0);
 \draw (0.25,0) .. controls (0.2,0.2) and (0.1,0.2) .. (0,0.2);
 \draw (0.25,0.4) .. controls (0.2,0.4) and (-0.15,0.5) .. (-0.15,0);
 \draw (-0.15,0) .. controls (-0.15,-0.25) and (0.15,-0.25) .. (0.15,0);
 \draw (0.15,0) .. controls (0.15,0.1) and (0.1,0.13) .. (0,0.13);
 \fill (0,0) circle (0.7mm);
\end{tikzpicture}
\ \text{or}\ 
\begin{tikzpicture}[baseline=-1mm]
 \draw[dotted,thick] (0,0) circle (5mm);
 \draw (0.25,0.4)--(0.25,0);
 \draw (0.25,0) .. controls (0.25,-0.4) and (-0.25,-0.4) .. (-0.25,0);
 \draw (-0.25,0) .. controls (-0.2,0.2) and (-0.1,0.2) .. (0,0.2);
 \draw (-0.25,0.4) .. controls (-0.2,0.4) and (0.15,0.5) .. (0.15,0);
 \draw (0.15,0) .. controls (0.15,-0.25) and (-0.15,-0.25) .. (-0.15,0);
 \draw (-0.15,0) .. controls (-0.15,0.1) and (-0.1,0.13) .. (0,0.13);
 \fill (0,0) circle (0.7mm);
\end{tikzpicture}\ .
\]
 In particular, ($\e^{-1}(\ell_{\sf p}),\e^{-1}(\ell_{\sf q})$) is a pair of conjugate arcs. For a lamination $L$ on $(S,M)$, we denote by $L_{\sf pq}$ the multi-set of elementary laminates obtained from $L$ by replacing exceptional laminates $\ell \in L$ with $\ell_{\sf p}$ and $\ell_{\sf q}$.

\begin{example}\label{ex2}
 In Example \ref{ex1}, $\ell_2$ and $\ell_5$ are exceptional, and $(\ell_2)_{\sf p} = \ell_3$, $(\ell_2)_{\sf q} = \ell_1$, $(\ell_5)_{\sf p} = \ell_4$ and $(\ell_5)_{\sf q} = \ell_6$. Thus we have the equalities
\[
 b_T(\ell_2)=b_T(\{(\ell_2)_{\sf p}, (\ell_2)_{\sf q}\})\ \ \text{and}\ \ b_T(\ell_5)=b_T(\{(\ell_5)_{\sf p}, (\ell_5)_{\sf q}\}).
\]
\end{example}

 In general, the same property as Example \ref{ex2} holds for arbitrary exceptional laminates.

\begin{lemma}\label{vecsep}
 Let $T$ be a tagged triangulation of $(S,M)$. For an exceptional laminate $\ell$ of $(S,M)$, we have
\[
 b_T(\ell)=b_T(\{\ell_{\sf p}, \ell_{\sf q}\}).
\]
\end{lemma}

\begin{proof}
 By Proposition \ref{excep}, there is a unique puncture $p$ enclosed by $\ell$. We only need to prove
\begin{equation}\label{eqpq}
 b_{\de,T}(\ell)=b_{\de,T}(\{\ell_{\sf p}, \ell_{\sf q}\})
\end{equation}
 for any $\de \in T$. If $\de \in T$ is not incident to $p$, then \eqref{eqpq} is clear. We assume that $\de$ is incident to $p$. Let $\de_1,\ldots,\de_m$ be tagged arcs of $T$ incident to $p$ winding clockwisely around $p$ such that the following conditions are satisfied (see Figure \ref{locp}):
\begin{itemize}
 \item $\ell$ crosses them at points $p_1,\ldots,p_m$ in this order;
 \item The segment of $\de_i$ from $p$ to $p_i$, that of $\de_{i+1}$ from $p$ to $p_{i+1}$, and that of $\ell$ from $p_i$ to $p_{i+1}$ form a contractible triangle.
\end{itemize}
Note that if these arcs contains a pair $(\de,\epsilon)$ of conjugate arcs, then we can choice the order of $\de$ and $\epsilon$.

\begin{figure}[htp]
\centering
\begin{tikzpicture}[baseline=5mm]
 \coordinate (0) at (0,0);
 \coordinate (p) at (0,0.8);
 \draw (-0.5,0) .. controls (-0.8,1.2) and (-0.3,1.5) .. (0,1.5);
 \draw (0.5,0) .. controls (0.8,1.2) and (0.3,1.5) ..node[above,pos=0.8]{$\ell$} (0,1.5);
 \draw (p)--node[left,pos=0.4]{$p_1$}(-1.3,0)node[left]{$\de_1$};
 \draw (p)--node[right,pos=0.4]{$p_m$}(1.3,0)node[right]{$\de_m$};
 \draw (p)--node[left,pos=0.35]{$p_2$}(-1.3,1.6)node[left]{$\de_2$};
 \fill (p) circle (0.7mm);
 \draw[semithick,dotted] (0.3,0.75) arc (-20:140:0.3);
 \node[below] at(p) {$p$};
\end{tikzpicture}
   \caption{Local configuration around a puncture $p$}
   \label{locp}
\end{figure}
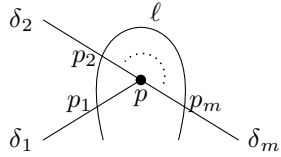

\noindent Moreover, $\de_1$ is different from $\de_m$. Indeed, if $\de_1 = \de_m$, considering triangles with a side $\de_1$, there is a tagged arc of $T$ incident to $p$ such that $\ell$ crosses it before $p_1$ or after $p_m$, a contradiction.

 The contributions to $b_{\de_i,T}(\ell)$ at $\de_i \cap \ell$ except at $p_i$ coincide with the contributions to $b_{\de_i,T}(\ell_{\sf p})$ and $b_{\de_i,T}(\ell_{\sf q})$ at them. We denote by $c_i(\ell)$ (resp, $c_i(\ell_{\sf p})$, $c_i(\ell_{\sf q})$) the contribution to $b_{\de_i,T}(\ell)$ (resp., $b_{\de_i,T}(\ell_{\sf p})$, $b_{\de_i,T}(\ell_{\sf q})$) at $p_i$ for $i \in \{1,\ldots,m\}$. To prove \eqref{eqpq}, we only need to show
\begin{equation}\label{ccc}
 c_i(\ell)=c_i(\ell_{\sf p})+c_i(\ell_{\sf q}).
\end{equation}

 First, we assume that neither $(\de_1,\de_2)$ nor $(\de_{m-1},\de_m)$ form a pair of conjugate arcs. Then it is easy to give the following values:
\[
 c_i(\ell) =  \left\{
     \begin{array}{ll}
   -1 & \text{if $i=1$},\\
   1 & \text{if $i=m$},\\
   0 & \text{if $i\neq 1, m$},
     \end{array} \right.\hspace{3mm}
 c_i(\ell_{\sf p}) =  \left\{
     \begin{array}{ll}
   -1 & \text{if $i=1$},\\
   0 & \text{if $i\neq 1$},
     \end{array} \right.\hspace{3mm}
 c_i(\ell_{\sf q}) =  \left\{
     \begin{array}{ll}
   1 & \text{if $i=m$},\\
   0 & \text{if $i\neq m$}.
     \end{array} \right.
\]
 Therefore, \eqref{ccc} holds.

 Second, we assume that $(\de_1,\de_2)$ is a pair of conjugate arcs tagged in the different ways at $p$, in which case $m=2$. Then by exchanging $\de_1$ and $\de_2$ if necessary (see Figure \ref{de1de2}), we have
\begin{equation}\label{conjci=}
 c_i(\ell_{\sf p}) =  \left\{
     \begin{array}{ll}
   c_1(\ell) & \text{if $i=1$},\\
   0 & \text{if $i=2$},
     \end{array} \right.\hspace{3mm}
 c_i(\ell_{\sf q}) =  \left\{
     \begin{array}{ll}
   0 & \text{if $i=1$},\\
   c_2(\ell) & \text{if $i=2$}.
     \end{array} \right.
\end{equation}
 Therefore, \eqref{ccc} holds.

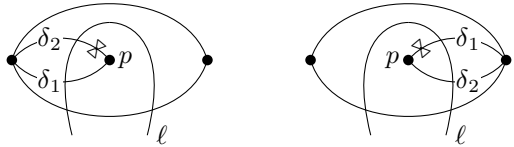
\begin{figure}[htp]
\centering
\begin{tikzpicture}[baseline=5mm]
 \coordinate (0) at (0,0);
 \coordinate (p) at (0,1);
 \coordinate (l) at (-1.3,1);
 \coordinate (r) at (1.3,1);
 \draw (l) .. controls (-1,2) and (1,2) .. (r);
 \draw (l) .. controls (-1,0) and (1,0) .. (r);
 \draw (l) to [out=50,in=130,relative] node[pos=0.85]{\rotatebox{50}{\footnotesize $\bowtie$}}  node[fill=white,inner sep=0.1,pos=0.4]{$\de_2$} (p);
 \draw (l) to [out=-50,in=-130,relative] node[fill=white,inner sep=0.1,pos=0.4]{$\de_1$} (p);
 \draw (-0.5,0) .. controls (-0.8,1.2) and (-0.3,1.5) .. (0,1.5);
 \draw (0.5,0) .. controls (0.8,1.2) and (0.3,1.5) ..node[right,pos=0]{$\ell$} (0,1.5);
 \fill (p) circle (0.7mm); \fill (l) circle (0.7mm); \fill (r) circle (0.7mm);
 \node[right] at(p) {$p$};
\end{tikzpicture}
   \hspace{10mm}
\begin{tikzpicture}[baseline=5mm]
 \coordinate (0) at (0,0);
 \coordinate (p) at (0,1);
 \coordinate (l) at (-1.3,1);
 \coordinate (r) at (1.3,1);
 \draw (l) .. controls (-1,2) and (1,2) .. (r);
 \draw (l) .. controls (-1,0) and (1,0) .. (r);
 \draw (r) to [out=50,in=130,relative]  node[fill=white,inner sep=0.1,pos=0.4]{$\de_2$} (p);
 \draw (r) to [out=-50,in=-130,relative] node[pos=0.85]{\rotatebox{-50}{\footnotesize $\bowtie$}} node[fill=white,inner sep=0.1,pos=0.4]{$\de_1$} (p);
 \draw (-0.5,0) .. controls (-0.8,1.2) and (-0.3,1.5) .. (0,1.5);
 \draw (0.5,0) .. controls (0.8,1.2) and (0.3,1.5) ..node[right,pos=0]{$\ell$} (0,1.5);
 \fill (p) circle (0.7mm); \fill (l) circle (0.7mm); \fill (r) circle (0.7mm);
 \node[left] at(p) {$p$};
\end{tikzpicture}
   \caption{Pairs $(\de_1,\de_2)$ of conjugate arcs tagged in the different ways at $p$ such that \eqref{conjci=} holds}
   \label{de1de2}
\end{figure}

 Finally, we assume that $(\de_1,\de_2)$ is a pair of conjugate arcs tagged in the same way at $p$. We define a set $\mathbb{M}$ as follows: If $(\de_{m-1},\de_m)$ is a pair of conjugate arcs, then $\mathbb{M}=\{m-1,m\}$; Otherwise, $\mathbb{M}=\{m\}$. Then we have $c_i(\ell_{\sf p}) = 0 = c_j(\ell_{\sf q})$ for $i \notin \{1,2\}$ and $j \notin \mathbb{M}$, and
\[
 c_i(\ell) =  \left\{
     \begin{array}{ll}
   c_i(\ell_{\sf p}) & \text{if $i\in\{1,2\}$},\\
   c_i(\ell_{\sf q}) & \text{if $i \in \mathbb{M}$},\\
   0 & \text{otherwise}.
     \end{array} \right.
\]
 Therefore, \eqref{ccc} holds. Moreover, it follows from the symmetry for the case that $(\de_{m-1},\de_m)$ is a pair of conjugate arcs tagged in the same way at $p$. Consequently, \eqref{eqpq} holds for any $\de \in T$.
\end{proof}

 For a lamination $L$ on $(S,M)$, we have decompositions
\begin{equation}\label{decomp}
 L = L_{\rm el} \sqcup L_{\rm ex} \sqcup L_{\rm cl} = L_{\rm nc} \sqcup L_{\rm cl},
\end{equation}
where $L_{\rm el}$ (resp., $L_{\rm ex}$, $L_{\rm cl}$) consists of all elementary (resp., exceptional, closed) laminates in $L$. For multi-sets $L$ and $L'$ of laminates of $(S,M)$, we define non-multi-sets
\[
 \e^{-1}(L) := \{\de \text{ : a tagged arc}\mid \e(\de) \in L\}\ \text{ and }\ L \setminus L' := \{\ell \in L \mid \ell \notin L'\}.
\]
The following properties are used to prove Theorem \ref{main} in Subsection \ref{subsecmain}.

\begin{proposition}\label{tagpar}
 Let $L$ be a lamination on $(S,M)$ with $L_{\rm cl} = \emptyset$. Then the following properties hold:
\begin{itemize}
 \item[(1)] $C_T(L) \subseteq C_T(L_{\sf pq})$.
 \item[(2)] $\e^{-1}(L_{\sf pq})$ is a partial tagged triangulation of $(S,M)$.
\end{itemize}
 Moreover, we take a set $U$ of tagged arcs of $(S,M)$ such that $T' = \e^{-1}(L_{\sf pq}) \sqcup U$ is a tagged triangulation. Then we have the equality
\begin{itemize}
 \item[(3)] $C_T(\e T') = C_T(L_{\sf pq} \sqcup \e U)$.
\end{itemize}
\end{proposition}

\begin{proof}
 (1) The assertion immediately follows from Lemma \ref{vecsep}.

 (2) By Proposition \ref{bij}(2), $\e^{-1}(L_{\rm el})$ is a partial tagged triangulation of $(S,M)$. Since $L$ is a lamination, any laminate in $(L_{\rm ex})_{\sf pq}$ is compatible with all laminates in $L_{\rm el} \setminus (L_{\rm ex})_{\sf pq}$. Then, by Proposition \ref{bij}(1), any tagged arc of $\e^{-1}((L_{\rm ex})_{\sf pq})$ is compatible with all tagged arcs of $\e^{-1}(L_{\rm el} \setminus (L_{\rm ex})_{\sf pq})$. Moreover, $\e^{-1}((L_{\rm ex})_{\sf pq})$ is a partial tagged triangulation since $(\e^{-1}(\ell_{\sf p}),\e^{-1}(\ell_{\sf q}))$ is a pair of conjugate arcs for $\ell \in L_{\rm ex}$. Therefore,
\[
 \e^{-1}(L_{\sf pq}) = \e^{-1}(L_{\rm el} \setminus (L_{\rm ex})_{\sf pq}) \sqcup \e^{-1}((L_{\rm ex})_{\sf pq})
\]
 is a partial tagged triangulation of $(S,M)$.

 (3) Since $L_{\sf pq}$ coincides with the multiplicity of $\e\e^{-1}(L_{\sf pq})$, we have the equalities
\[
C_T(L_{\sf pq} \sqcup \e U) = C_T(\e\e^{-1}(L_{\sf pq}) \sqcup \e U) = C_T(\e T').
\qedhere
\].
\end{proof}

\subsection{Shear coordinates and Dehn twists}

 We consider the Dehn twist along a closed laminate and its effect on shear coordinates. In this subsection, we fix an ideal or tagged triangulation $T$, a closed laminate $\ell_c$ of $(S,M)$ and its direction. We denote by $\mathsf{T}_{\ell_c}$ the Dehn twist of $(S,M)$ along $\ell_c$ defined from the direction of $\ell_c$ as follows:
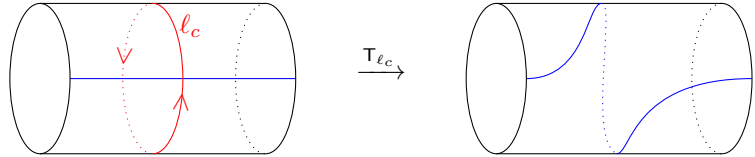
\[
\begin{tikzpicture}[baseline=0mm]
 \coordinate (0) at (0,0);
 \draw (0,1)--(3,1) (0,-1)--(3,-1); \draw[blue] (0.4,0)--(3.4,0);
 \draw[red] (1.5,1) arc [start angle = 90, end angle = -90, x radius=4mm, y radius=10mm]  node[pos=0.6]{\rotatebox{85}{$>$}};
 \draw[red,dotted] (1.5,1) arc [start angle = 90, end angle = 270, x radius=4mm, y radius=10mm]  node[pos=0.4]{\rotatebox{-95}{$>$}};
 \draw (0) circle [x radius=4mm, y radius=10mm];
 \draw (3,1) arc [start angle = 90, end angle = -90, x radius=4mm, y radius=10mm];
 \draw[dotted] (3,1) arc [start angle = 90, end angle = 270, x radius=4mm, y radius=10mm];
 \node[red] at(2,0.7) {$\ell_c$};
\end{tikzpicture}
 \hspace{7mm} \xrightarrow{\mathsf{T}_{\ell_c}} \hspace{7mm}
\begin{tikzpicture}[baseline=0mm]
 \coordinate (0) at (0,0);
 \draw (0,1)--(3,1) (0,-1)--(3,-1);
 \draw[blue] (0.4,0) .. controls (1.2,0) and (1.2,1) .. (1.4,1);
 \draw[blue] (1.6,-1) .. controls (1.8,-1) and (1.8,0) .. (3.4,0);
 \draw[blue,dotted] (1.4,1) .. controls (1.6,1) and (1.2,-0.8) .. (1.6,-1);
 \draw (0) circle [x radius=4mm, y radius=10mm];
 \draw (3,1) arc [start angle = 90, end angle = -90, x radius=4mm, y radius=10mm];
 \draw[dotted] (3,1) arc [start angle = 90, end angle = 270, x radius=4mm, y radius=10mm];
\end{tikzpicture}
\]

 The aim of this subsection is to prove the following.

\begin{theorem}\label{Dehn}
 Let $\ell_c$ be a closed laminate and $\ell$ a laminate of $(S,M)$ intersecting with $\ell_c$, and let $\de \in T$. Then there is $m' \in \bZ_{\ge 0}$ such that for any $m \ge m'$, we have
\[
 b_{\de,T}(\mathsf{T}_{\ell_c}^m(\ell)) = b_{\de,T}(\mathsf{T}_{\ell_c}^{m'}(\ell)) + (m-m') \#(\ell \cap \ell_c) b_{\de,T}(\ell_c).
\]
\end{theorem}

 First, we assume that $(S,M)$ is an annulus without punctures and $T$ is its ideal triangulation consisting of arcs $\tau_1,\ldots,\tau_r$ crossing $\ell_c$ in order of occurrence along $\ell_c$ (we can have $\tau_{i}=\tau_j$ even if $i \neq j$), that is,
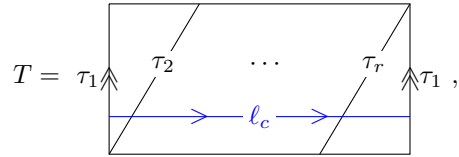
\[
 T=
\begin{tikzpicture}[baseline=-0.5mm]
 \coordinate (lu) at (-2,1); \coordinate (ru) at (2,1); \coordinate (ld) at (-2,-1); \coordinate (rd) at (2,-1);
 \draw (lu)--(ru)--node[right]{$\tau_1$}node[pos=0.5]{\rotatebox{90}{$\gg$}}(rd)--(ld)--node[left]{$\tau_1$}node[pos=0.5]{\rotatebox{90}{$\gg$}}(lu);
 \draw (ld)--node[fill=white,inner sep=2,pos=0.6]{$\tau_2$}(-0.8,1); \draw (ru)--node[fill=white,inner sep=2,pos=0.4]{$\tau_r$}(0.8,-1);
 \node at(0.1,0.2) {$\cdots$};
 \draw[blue] (-2,-0.5)--node[fill=white,inner sep=2]{$\ell_c$}node[pos=0.3]{$>$}node[pos=0.7]{$>$}(2,-0.5);
\end{tikzpicture},
\]
 where two vertical lines $\tau_1$ are identified. Any elementary laminate $\ell$ of $(S,M)$ intersects with $\ell_c$ at most once since they intersect in a minimal number of points. We assume that $\ell$ intersects with $\ell_c$. We define the direction of $\ell$ as crossing $\ell_c$ from left to right:
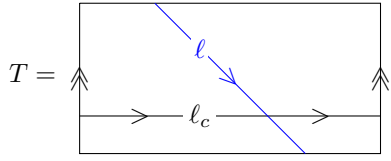
\[
 T=
\begin{tikzpicture}[baseline=-0.5mm]
 \coordinate (lu) at (-2,1); \coordinate (ru) at (2,1); \coordinate (ld) at (-2,-1); \coordinate (rd) at (2,-1);
 \draw (lu)--(ru)--node[pos=0.5]{\rotatebox{90}{$\gg$}}(rd)--(ld)--node[pos=0.5]{\rotatebox{90}{$\gg$}}(lu);
 \draw (-2,-0.5)--node[fill=white,inner sep=2,pos=0.4]{$\ell_c$}node[pos=0.2]{$>$}node[pos=0.8]{$>$}(2,-0.5);
 \draw[blue] (-1,1)--node[pos=0.5]{\rotatebox{-45}{$>$}}node[fill=white,inner sep=1,pos=0.3]{$\ell$}(1,-1);
\end{tikzpicture}
\]
 Let $s \in \{1,\ldots,r\}$ such that the starting point of $\ell$ is on the triangle of $T$ with sides $\tau_{s-1}$ and $\tau_s$, where $\tau_{r+i}=\tau_i$. In particular, $\ell$ intersects at least one of $\tau_{s-1}$ and $\tau_s$. Thus $\ell$ intersects with the $t_{\ell}$ ($\in \bZ_{\ge 1}$) diagonals either $\tau_s,\tau_{s+1},\ldots,\tau_{s+t_{\ell}-1}$ or $\tau_{s-1},\tau_{s-2},\ldots,\tau_{s-t_{\ell}}$ of $T$ in order. In the former (resp., latter) case, we say that $\ell$ {\it intersects with $T$ in ascending} (resp., {\it descending}) {\it order}:
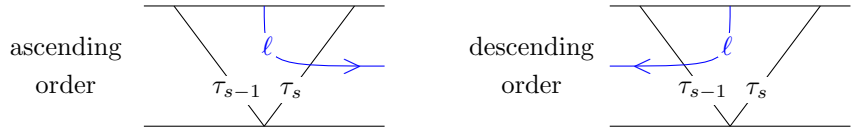
\[
\begin{tikzpicture}[baseline=0mm,scale=0.8]
 \coordinate (lu) at (-2,1); \coordinate (ru) at (2,1); \coordinate (ld) at (-2,-1); \coordinate (rd) at (2,-1);
 \draw (lu)--(ru) (rd)--(ld);
 \draw (-1.5,1)--node[fill=white,inner sep=2,pos=0.7]{$\tau_{s-1}$}(0,-1);
 \draw (1.5,1)--node[fill=white,inner sep=2,pos=0.7]{$\tau_s$}(0,-1);
 \draw[blue] (0,1) .. controls (0,0) ..node[pos=0.9]{$>$}node[fill=white,inner sep=1,pos=0.3]{$\ell$}(2,0);
 \node at(-3.3,0.3){ascending}; \node at(-3.3,-0.3){order};
\end{tikzpicture}
 \hspace{10mm}
\begin{tikzpicture}[baseline=0mm,scale=0.8]
 \coordinate (lu) at (-2,1); \coordinate (ru) at (2,1); \coordinate (ld) at (-2,-1); \coordinate (rd) at (2,-1);
 \draw (lu)--(ru) (rd)--(ld);
 \draw (-1.5,1)--node[fill=white,inner sep=2,pos=0.7]{$\tau_{s-1}$}(0,-1);
 \draw (1.5,1)--node[fill=white,inner sep=2,pos=0.7]{$\tau_s$}(0,-1);
 \draw[blue] (0,1) .. controls (0,0) ..node[pos=0.9]{$<$}node[fill=white,inner sep=1,pos=0.3]{$\ell$}(-2,0);
 \node at(-3.3,0.3){descending}; \node at(-3.3,-0.3){order};
\end{tikzpicture}
\]

\begin{proposition}\label{annu}
 Let $(S,M)$ be an annulus without punctures and $\ell$ an elementary laminate of $(S,M)$ intersecting with $\ell_c$.
\begin{itemize}
 \item[(1)] If $\ell$ intersects with $T$ in ascending order, then so is $\mathsf{T}_{\ell_c}(\ell)$ and for $\de \in T$, we have
\[
 b_{\de,T}(\mathsf{T}_{\ell_c}(\ell)) = b_{\de,T}(\ell) + b_{\de,T}(\ell_c).
\]
 If $\ell$ and $\mathsf{T}_{\ell_c}(\ell)$ intersect with $T$ in descending order, then for $\de \in T$, we have
\[
 b_{\de,T}(\mathsf{T}_{\ell_c}(\ell)) = b_{\de,T}(\ell) - b_{\de,T}(\ell_c).
\]
 \item[(2)] There is $m \gg 0$ such that $\mathsf{T}_{\ell_c}^m(\ell)$ intersects with $T$ in ascending order.
\end{itemize}
\end{proposition}

\begin{proof}
 (1) We only prove the first assertion since the proof of the second assertion is similar. Suppose that $\ell$ intersects with $T$ in ascending order. If $t_{\ell} >1 $, then the assertion holds since $\mathsf{T}_{\ell_c}$ only transforms $\ell$ around an intersection point of $\ell$ and $\ell_c$ as follows:
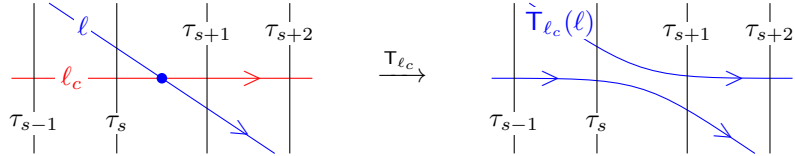
\[
\begin{tikzpicture}[baseline=-0.5mm]
 \coordinate (lu) at (-1.5,1); \coordinate (ru) at (2,1); \coordinate (ld) at (-2,-1); \coordinate (rd) at (1.5,-1);
 \draw[blue] (lu)--node[pos=0.85]{\rotatebox{-30}{$>$}}node[fill=white,inner sep=1,pos=0.15]{$\ell$}(rd);
 \draw[red] (-2,0)--node[fill=white,inner sep=2,pos=0.2]{$\ell_c$}node[pos=0.8]{$>$}(2,0);
 \fill[blue] (0,0) circle (0.7mm); \node[blue] at(0.05,0.2){};
 \draw (-0.6,1)--node[fill=white,inner sep=1,pos=0.8]{$\tau_s$}(-0.6,-1);
 \draw (0.6,1)--node[fill=white,inner sep=1,pos=0.2]{$\tau_{s+1}$}(0.6,-1);
 \draw (-1.7,1)--node[fill=white,inner sep=1,pos=0.8]{$\tau_{s-1}$}(-1.7,-1);
 \draw (1.7,1)--node[fill=white,inner sep=1,pos=0.2]{$\tau_{s+2}$}(1.7,-1);
\end{tikzpicture}
 \hspace{7mm} \xrightarrow{\mathsf{T}_{\ell_c}} \hspace{7mm}
\begin{tikzpicture}[baseline=-0.5mm]
 \coordinate (lu) at (-1.5,1); \coordinate (ru) at (2,1); \coordinate (ld) at (-2,-1); \coordinate (rd) at (1.5,-1);
 \draw[blue] (lu) .. controls (0,0) ..node[pos=0.85]{$>$}node[fill=white,inner sep=1,pos=0.1]{$\mathsf{T}_{\ell_c}(\ell)$}(2,0);
 \draw[blue] (-2,0) .. controls (0,0) ..node[pos=0.15]{$>$}node[pos=0.9]{\rotatebox{-30}{$>$}}(rd);
 \draw (-0.6,1)--node[fill=white,inner sep=1,pos=0.8]{$\tau_s$}(-0.6,-1);
 \draw (0.6,1)--node[fill=white,inner sep=1,pos=0.2]{$\tau_{s+1}$}(0.6,-1);
 \draw (-1.7,1)--node[fill=white,inner sep=1,pos=0.8]{$\tau_{s-1}$}(-1.7,-1);
 \draw (1.7,1)--node[fill=white,inner sep=1,pos=0.2]{$\tau_{s+2}$}(1.7,-1);
\end{tikzpicture}
\]
 If $t_{\ell} = 1$, then $\ell$ and $\mathsf{T}_{\ell_c}(\ell)$ are given as follows:
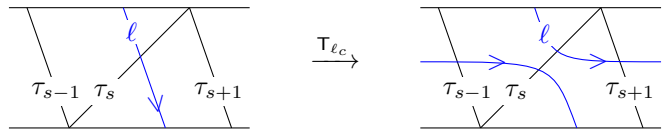
\[
\begin{tikzpicture}[baseline=0mm,scale=0.8]
 \coordinate (lu) at (-2,1); \coordinate (ru) at (2,1); \coordinate (ld) at (-2,-1); \coordinate (rd) at (2,-1);
 \draw (lu)--(ru) (rd)--(ld);
 \draw (-1.7,1)--node[fill=white,inner sep=2,pos=0.7]{$\tau_{s-1}$}(-1,-1);
 \draw (1,1)--node[fill=white,inner sep=2,pos=0.7]{$\tau_s$}(-1,-1);
 \draw (1,1)--node[fill=white,inner sep=2,pos=0.7]{$\tau_{s+1}$}(1.7,-1);
 \draw[blue] (-0.1,1)--node[pos=0.8]{\rotatebox{-80}{$>$}}node[fill=white,inner sep=1,pos=0.2]{$\ell$}(0.6,-1);
\end{tikzpicture}
 \hspace{7mm} \xrightarrow{\mathsf{T}_{\ell_c}} \hspace{7mm}
\begin{tikzpicture}[baseline=0mm,scale=0.8]
 \coordinate (lu) at (-2,1); \coordinate (ru) at (2,1); \coordinate (ld) at (-2,-1); \coordinate (rd) at (2,-1);
 \draw (lu)--(ru) (rd)--(ld);
 \draw (-1.7,1)--node[fill=white,inner sep=2,pos=0.7]{$\tau_{s-1}$}(-1,-1);
 \draw (1,1)--node[fill=white,inner sep=2,pos=0.7]{$\tau_s$}(-1,-1);
 \draw (1,1)--node[fill=white,inner sep=2,pos=0.7]{$\tau_{s+1}$}(1.7,-1);
 \draw[blue] (-0.1,1) .. controls (0.2,0.1) ..node[pos=0.75]{$>$}node[fill=white,inner sep=1,pos=0.2]{$\ell$}(2,0.1);
 \draw[blue] (0.6,-1) .. controls (0.2,0.1) ..node[pos=0.75]{$>$}(-2,0.1);
\end{tikzpicture}
\]
 Then the assertion is directly given by enumerating their shear coordinates.

 (2) Suppose that $\ell$ intersects with $T$ in descending order. If $t_{\ell} < r$, then $\mathsf{T}_{\ell_c}(\ell)$ intersects with $T$ in ascending order. If $t_{\ell} \ge r$, then $\mathsf{T}_{\ell_c}(\ell)$ intersects with $T$ in descending order and $t_{\mathsf{T}_{\ell_c}(\ell)}=t_{\ell}-r$. By the induction, the assertion holds.
\end{proof}

 Next, we consider an arbitrary marked surface $(S,M)$ and its ideal triangulation $T$. For $\ell_c$ in Theorem \ref{Dehn}, we construct an annulus $(S_{\ell_c},M_{\ell_c})$ and its triangulation $T_{\ell_c}$ as follows: Let $\tau_1,\ldots,\tau_r$ be the arcs of $T$ crossing $\ell_c$ in order of occurrence along $\ell_c$ ($\tau_{i}$ and $\tau_j$ can be the same even if $i \neq j$). Hence $\ell_c$ crosses $r$ triangles $\triangle_1,\ldots,\triangle_r$ in this order. For $i \in \{1,\ldots,r\}$, let $\triangle'_i$ be a copy of the triangle $\triangle_i$, hence $\triangle'_i$ has the sides $\tau_i$ and $\tau_{i+1}$. Then an annulus $(S_{\ell_c},M_{\ell_c})$ and its triangulation $T_{\ell_c}$ are obtained by gluing $\triangle'_1,\ldots,\triangle'_r$ along the edges $\tau_i$, that is,
\begin{equation}\label{Tellc}
 T_{\ell_c}=
\begin{tikzpicture}[baseline=-0.5mm]
 \coordinate (lu) at (-2,1); \coordinate (ru) at (2,1); \coordinate (ld) at (-2,-1); \coordinate (rd) at (2,-1);
 \draw (lu)--(ru)--node[right]{$\tau_1$}node[pos=0.5]{\rotatebox{90}{$\gg$}}(rd)--(ld)--node[left]{$\tau_1$}node[pos=0.5]{\rotatebox{90}{$\gg$}}(lu);
 \draw (ld)--node[right,pos=0.6]{$\tau_2$}(-0.8,1); \draw (ru)--node[left,pos=0.4]{$\tau_r$}(0.8,-1);
 \node at(0.1,0.2) {$\cdots$}; \node at(-1.6,0.4) {$\triangle'_1$}; \node at(1.65,-0.25) {$\triangle'_r$};
 \draw[blue] (-2,-0.5)--node[fill=white,inner sep=2]{$\ell_c$}node[pos=0.3]{$>$}node[pos=0.7]{$>$}(2,-0.5);
\end{tikzpicture}
\end{equation}
in $(S_{\ell_c},M_{\ell_c})$, where two vertical lines $\tau_1$ are identified. In particular, if $\tau_i$ is inside a self-folded triangle of $T$, then the corresponding triangles are given by
\[
\begin{tikzpicture}[baseline=-3mm]
 \coordinate (0) at (0,0);
 \coordinate (1) at (0,-1);
 \draw (0) node[left]{$p$} to node[right,pos=0.3]{$\tau_i$} (1);
 \draw (0,0.4) node[above]{$\tau_{i-1}=\tau_{i+1}$} ..controls(0.5,0.4)and(1,-0.2).. (1);
 \draw (0,0.4) ..controls(-0.5,0.4)and(-1,-0.2).. (1);
 \fill(0) circle (0.7mm); \fill (1) circle (0.7mm);
 \draw[blue] (-1.2,-0.6)--node[fill=white,inner sep=1,pos=0.15]{$\ell_c$}node[pos=0.85]{$>$}(1.2,-0.6);
\end{tikzpicture}
 \text{in $T$}
 \hspace{5mm} \longrightarrow \hspace{5mm}
\begin{tikzpicture}[baseline=-1mm]
 \coordinate (0) at (0,-1); \coordinate (l) at (-1,0.7); \coordinate (r) at (1,0.7); \coordinate (u) at (0,0.7);
 \draw (0)--node[left,pos=0.6]{$\tau_{i-1}$}(l)--(r)--node[right,pos=0.4]{$\tau_{i+1}$}(0)--node[fill=white,inner sep=2,pos=0.6]{$\tau_i$}(u);
 \draw[blue] (-1.2,-0.5)--node[fill=white,inner sep=1,pos=0.15]{$\ell_c$}node[pos=0.85]{$>$}(1.2,-0.5);
\end{tikzpicture}
\ \ \text{in $T_{\ell_c}$}.
\]
 For a laminate $\ell$ of $(S,M)$ intersecting with $\ell_c$, let $q \in \ell \cap \ell_c$ and $\ell_q$ a laminate of $(S_{\ell_c},M_{\ell_c})$ corresponding to the connected segment of $\ell$ in $T_{\ell_c}$ containing $q$ as follows:
\[
\begin{tikzpicture}[baseline=-1mm]
 \coordinate (lu) at (-2,0.8); \coordinate (ru) at (2,0.8); \coordinate (ld) at (-2,-0.8); \coordinate (rd) at (2,-0.8);
 \draw (lu)--(ru)--node[pos=0.5]{\rotatebox{90}{$\gg$}}(rd)--(ld)--node[pos=0.5]{\rotatebox{90}{$\gg$}}(lu);
 \draw[red] (-2,-0.5)--node[fill=white,inner sep=2,pos=0.4]{$\ell_c$}node[pos=0.8]{$>$}(2,-0.5);
 \draw (lu)--(-2,1.3); \draw (ru)--(2,1.3); \draw (ld)--(-2,-1.3); \draw (rd)--(2,-1.3);
 \draw[blue] (-1.8,1.3)--(-1,-1.3); \draw[blue] (1.6,1.3)--(0,-1.3);
 \draw[blue] (-1.5,1.3)..controls(-1,0)and(0.5,0)..(1.2,1.3)node[left]{$\ell$};
 \fill[blue](-1.25,-0.5) circle (0.7mm); \fill[blue](0.5,-0.5) circle (0.7mm);
 \node[blue] at(-1.05,-0.35) {$q$}; \node[blue] at(0.33,-0.33) {$q'$};
\end{tikzpicture}
\ \text{in $T$}
 \hspace{5mm} \longrightarrow \hspace{5mm}
\begin{tikzpicture}[baseline=-1mm]
 \coordinate (lu) at (-2,0.8); \coordinate (ru) at (2,0.8); \coordinate (ld) at (-2,-0.8); \coordinate (rd) at (2,-0.8);
 \draw (lu)--(ru)--node[pos=0.5]{\rotatebox{90}{$\gg$}}(rd)--(ld)--node[pos=0.5]{\rotatebox{90}{$\gg$}}(lu);
 \draw[red] (-2,-0.5)--node[fill=white,inner sep=2,pos=0.4]{$\ell_c$}node[pos=0.8]{$>$}(2,-0.5);
 \draw[blue] (-1.6,0.8)--node[fill=white,inner sep=2,pos=0.4]{$\ell_q$}(-1.2,-0.8); \draw[blue] (1.3,0.8)--node[fill=white,inner sep=2,pos=0.4]{$\ell_{q'}$}(0.3,-0.8);
 \fill[blue](-1.28,-0.5) circle (0.7mm); \fill[blue](0.5,-0.5) circle (0.7mm);
 \node[blue] at(-1.05,-0.35) {$q$}; \node[blue] at(0.33,-0.33) {$q'$};
\end{tikzpicture}
\ \text{in $T_{\ell_c}$}
\]

\begin{proposition}\label{Dehn?}
 Let $T$ be an ideal triangulation of $(S,M)$, $\g \in T$, and $\ell$ a laminate of $(S,M)$ intersecting with $\ell_c$. We assume that for all $q \in \ell \cap \ell_c$, $\ell_q$ intersects with $T_{\ell_c}$ in ascending order. Then we have
\[
 b_{\g,T}(\mathsf{T}_{\ell_c}(\ell)) = b_{\g,T}(\ell) + \#(\ell \cap \ell_c) b_{\g,T}(\ell_c).
\]
\end{proposition}

\begin{proof}
 The proof is divided into the following three cases (1)--(3).

 (1) If $\g$ is not a side of triangles of $T$, then the assertion is clear.

 (2) We assume that $\g$ is a side of some triangle of $T$ and $\g \cap \ell_c \neq \emptyset$. If $\g$ is not inside a self-folded triangle of $T$, then the construction of $T_{\ell_c}$ preserves the quadrilateral surrounding $\g$. Therefore, we have
{\setlength\arraycolsep{0.5mm}
\begin{eqnarray*}
 b_{\g,T}(\mathsf{T}_{\ell_c}(\ell)) - b_{\g,T}(\ell) &=& \sum_{q \in \ell \cap \ell_c} \sum_{1 \le i \le r, \tau_i = \g} \Bigl(b_{\tau_i,T_{\ell_c}}(\mathsf{T}_{\ell_c}(\ell_q)) - b_{\tau_i,T_{\ell_c}}(\ell_q)\Bigr)\\
 &\overset{\ref{annu}(1)}{=}& \sum_{q \in \ell \cap \ell_c} \sum_{1 \le i \le r, \tau_i = \g} b_{\tau_i,T_{\ell_c}}(\ell_c) = \#(\ell \cap \ell_c) b_{\g,T}(\ell_c).
\end{eqnarray*}}
 Suppose that $\g$ is inside a self-folded triangle $\{\g,\g'\}$ of $T$ enclosing a puncture $p$. Recall that $b_{\g,T}(\ell) = b_{\g',T}(\ell^{(p)})$, where $\ell^{(p)}$ is a laminate obtained from $\ell$ by changing the directions of its spirals at $p$ if they exist (see Subsection \ref{lami}). Thus the assertion follows from the previous case if for all $q \in \ell^{(p)} \cap \ell_c$, $(\ell^{(p)})_q$ intersects with $T_{\ell_c}$ in ascending order. This is checked as follows: If no ends of $\ell$ are spirals at $p$, then $\ell^{(p)}=\ell$, hence it is clear. Otherwise, $\ell \neq \ell^{(p)}$, and $\ell \cap \ell_c$ and $\ell^{(p)} \cap \ell_c$ are identified in the natural way. Then $\ell_q$ and $(\ell^{(p)})_q$ are only different in that their ends around $p$ are given by
\[
\begin{tikzpicture}[baseline=0mm]
 \coordinate (0) at (0,-0.7); \coordinate (l) at (-1,0.7); \coordinate (r) at (1,0.7); \coordinate (u) at (0,0.7);
 \draw (0)--(l)--(r)--(0)--node[fill=white,inner sep=2,pos=0.3]{$\g$}(u);
 \draw[blue] (-0.3,0.7) .. controls (0,0) ..node[pos=0.7,above]{$\ell'$}(1.2,0);
 \node[above] at(u){$p$};
 \draw (-1.2,0.7)--(l) (r)--(1.2,0.7) (-1.2,-0.7)--(1.2,-0.7);
\end{tikzpicture}
 \hspace{5mm}\text{and}\hspace{5mm}
\begin{tikzpicture}[baseline=0mm]
 \coordinate (0) at (0,-0.7); \coordinate (l) at (-1,0.7); \coordinate (r) at (1,0.7); \coordinate (u) at (0,0.7);
 \draw (0)--(l)--(r)--(0)--node[fill=white,inner sep=2,pos=0.5]{$\g$}(u);
 \draw[blue] (0.3,0.7) .. controls (0.6,0) ..node[pos=0.7,below]{$\ell''$}(1.2,0);
 \node[above] at(u){$p$};
 \draw (-1.2,0.7)--(l) (r)--(1.2,0.7) (-1.2,-0.7)--(1.2,-0.7);
\end{tikzpicture}
 \hspace{5mm}\text{in $T_{\ell_c}$, where $\{\ell',\ell''\}=\{\ell_q,(\ell^{(p)})_q\}$.}
\]
 Therefore, $(\ell^{(p)})_q$ also intersects with $T_{\ell_c}$ in ascending order. 

 (3) We assume that $\g$ is a side of some triangle $\triangle_i$ of $T$ and $\g \cap \ell_c = \emptyset$. Then we prove $b_{\g,T}(\mathsf{T}_{\ell_c}(\ell)) = b_{\g,T}(\ell)$. In this case, $\g$ is not inside a self-folded triangle of $T$. Indeed, if $\g$ is inside a self-folded triangle of $T$, then $\g$ is either $\tau_i$ or $\tau_{i+1}$, hence it is a contradiction. Therefore, there is the quadrilateral $\squareslash$ surrounding $\g$ of $T$. Since for all $q \in \ell \cap \ell_c$, $\ell_q$ intersects with $T_{\ell_c}$ in ascending order, the Dehn twist $\mathsf{T}_{\ell_c}$ affects $\ell \cap \squareslash$ as follows:
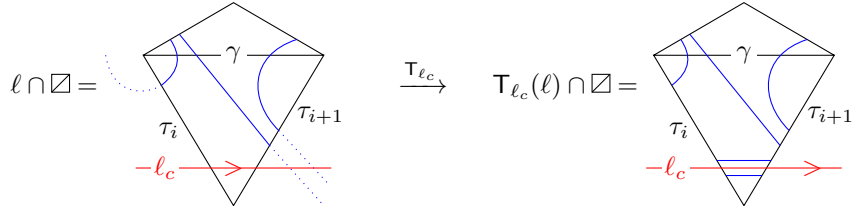
\[
 \ell \cap \squareslash =
\begin{tikzpicture}[baseline=-5mm]
 \coordinate (u) at (0,0.7); \coordinate (d) at (0,-2); \coordinate (l) at (-1.2,0); \coordinate (r) at (1.2,0);
 \draw (d)--node[left]{$\tau_i$}(l)--node[fill=white,inner sep=2]{$\g$}(r)--node[right,pos=0.4]{$\tau_{i+1}$}(d) (l)--(u)--(r);
 \draw[blue]($(l)!0.4!(u)$)--($(r)!0.6!(d)$); \draw[blue,dotted]($(r)!0.6!(d)$)--(1.2,-2);
 \draw[blue]($(l)!0.3!(u)$)..controls(-0.7,0)and(-0.7,-0.2)..($(l)!0.2!(d)$);
 \draw[blue,dotted](-1.7,0)..controls(-1.6,-0.6)and(-1.1,-0.5)..($(l)!0.2!(d)$);
 \draw[blue]($(r)!0.3!(u)$)..controls(0.2,0)and(0.2,-0.6)..($(r)!0.5!(d)$);
 \draw[blue,dotted]($(r)!0.5!(d)$)--(1.25,-1.7);
 \draw[red] (-1.3,-1.5)--node[fill=white,inner sep=1,pos=0.15]{$\ell_c$}node[pos=0.5]{$>$}(1.3,-1.5);
\end{tikzpicture}
 \hspace{5mm} \xrightarrow{\mathsf{T}_{\ell_c}} \hspace{5mm}
 \mathsf{T}_{\ell_c}(\ell) \cap \squareslash =
\begin{tikzpicture}[baseline=-5mm]
 \coordinate (u) at (0,0.7); \coordinate (d) at (0,-2); \coordinate (l) at (-1.2,0); \coordinate (r) at (1.2,0);
 \draw (d)--node[left]{$\tau_i$}(l)--node[fill=white,inner sep=2]{$\g$}(r)--node[right,pos=0.4]{$\tau_{i+1}$}(d) (l)--(u)--(r);
 \draw[blue]($(l)!0.4!(u)$)--($(r)!0.6!(d)$);
 \draw[blue]($(l)!0.3!(u)$)..controls(-0.7,0)and(-0.7,-0.2)..($(l)!0.2!(d)$);
 \draw[blue]($(r)!0.3!(u)$)..controls(0.2,0)and(0.2,-0.6)..($(r)!0.5!(d)$);
 \draw[blue]($(l)!0.7!(d)$)--($(r)!0.7!(d)$); \draw[blue]($(l)!0.8!(d)$)--($(r)!0.8!(d)$);
 \draw[red] (-1.3,-1.5)--node[fill=white,inner sep=1,pos=0.15]{$\ell_c$}node[pos=0.85]{$>$}(1.3,-1.5);
\end{tikzpicture}
\]
 Therefore, it gives $b_{\g,T}(\mathsf{T}_{\ell_c}(\ell)) = b_{\g,T}(\ell)$.
\end{proof}

 We are ready to prove Theorem \ref{Dehn}.

\begin{proof}[Proof of Theorem \ref{Dehn}]
 First of all, we prove Theorem \ref{Dehn} for an ideal triangulation $T$. For $q \in \ell \cap \ell_c$, by Proposition \ref{annu}(2) there exists $m_q \in \bZ_{\ge 0}$ such that $\mathsf{T}_{\ell_c}^{m_q}(\ell_q)$ intersects with $T_{\ell_c}$ in ascending order. Thus for
\[
 m \ge m' := \max_{q \in \ell \cap \ell_c}\{m_q\},
\]
$\mathsf{T}_{\ell_c}^m(\ell_q)$ intersects with $T_{\ell_c}$ in ascending order for each $q \in \ell \cap \ell_c$. Therefore, by Theorem \ref{Dehn?}, we have
{\setlength\arraycolsep{0.5mm}
\begin{eqnarray*}
 b_{\g,T}(\mathsf{T}_{\ell_c}^{m+1}(\ell)) &=& b_{\g,T}(\mathsf{T}_{\ell_c}^m(\ell)) + \#(\ell \cap \ell_c) b_{\g,T}(\ell_c) = \cdots\\
 &=& b_{\g,T}(\mathsf{T}_{\ell_c}^{m'}(\ell)) + (m+1-m')\#(\ell \cap \ell_c) b_{\g,T}(\ell_c).
\end{eqnarray*}}\par
 For an arbitrary tagged triangulation $T$, we recall that there is a unique ideal triangulation $T^0$ satisfying $T'=\iota(T^0)$, where $T'$ is obtained from $T$ by simultaneous changing all tags at some punctures if necessary (see Subsection \ref{lami} for details). Then the shear coordinate $b_{\g,T}(\ell)$ of a laminate $\ell$ is equal to $b_{\g^0,T^0}(\ell')$, where $\ell'$ is a laminate obtained from $\ell$ by changing the directions of its spirals at these punctures if they exist. Since the change of directions of its spirals and the Dehn twist $\mathsf{T}_{\ell_c}$ are compatible, the proof of Theorem \ref{Dehn} comes down to the case of ideal triangulations.
\end{proof}

\subsection{Proof of Theorem \ref{main}}\label{subsecmain}

 In this subsection, we fix a tagged triangulation $T$ of $(S,M)$. By Theorem \ref{thurston}, to prove the first assertion of Theorem \ref{main}, we only need to show that for each lamination $L$ on $(S,M)$,
\begin{equation}\label{vectin}
 b_T(L) \in \overline{\bigcup_{T' \in \bT}C_T(\e(T'))}.
\end{equation}
 To prove \eqref{vectin}, we need some preparation. We have decompositions \eqref{decomp} of $L$. By Proposition \ref{tagpar}(2), $\e^{-1}((L_{\rm nc})_{\sf pq})$ is a partial tagged triangulation of $(S,M)$. Then we take a set $U$ of tagged arcs of $(S,M)$ such that $T_L := \e^{-1}((L_{\rm nc})_{\sf pq}) \sqcup U$ is a tagged triangulation.

\begin{lemma}\label{clU}
 Let $L$ be a lamination on $(S,M)$.
\begin{itemize}
 \item[(1)] Any closed laminate $\ell$ in $L_{\rm cl}$ does not intersect with tagged arcs of $\e^{-1}((L_{\rm nc})_{\sf pq})$, but it intersects with at least one tagged arc of $U$.
 \item[(2)] $C_T(\e\mathsf{T}_{\ell}(T_L)) = C_T((L_{\rm nc})_{\sf pq} \sqcup \e \mathsf{T}_{\ell}(U))$.
\end{itemize}
\end{lemma}

\begin{proof}
 $(1)$ Since $L$ is a lamination, any closed laminate $\ell$ in $L_{\rm cl}$ does not intersect with all laminates in $L_{\rm nc}$. Thus $\ell$ does not intersect with all tagged arcs of $\e^{-1}((L_{\rm nc})_{\sf pq})$ since $\e^{-1}$ and $(-)_{\sf pq}$ transform laminates just around the marked points of $(S,M)$. Moreover, since $T_L$ is a tagged triangulation and $\ell$ is not contractible, $\ell$ intersects with at least one tagged arc of $T_L$, hence it is of $U$.

 $(2)$ By $(1)$, we have $\mathsf{T}_{\ell}(T_L) = \e^{-1}((L_{\rm nc})_{\sf pq}) \sqcup \mathsf{T}_{\ell}(U)$ and it is a tagged triangulation. The desired equality is given by Proposition \ref{tagpar}(3).
\end{proof}

 Let $\ell_1,\ldots,\ell_t$ be all distinct closed laminates in $L_{\rm cl}$ and $n_i$ the multiplicity of $\ell_i$ in $L_{\rm cl}$ for $i \in \{1,\ldots,t\}$. By Lemma \ref{clU}(1), $N_i := \sum_{\epsilon \in U}\#(\ell_i \cap \epsilon)$ is not zero. In particular, $N_i$ is equal to $\sum_{\epsilon \in U}\#(\ell_i \cap \e(\epsilon))$. We fix the direction of $\ell_i$ for $1 \le i \le t$ and consider the Dehn twists $\mathsf{T}_{\ell_i}$. Since $\ell_i$ are not intersect, $\mathsf{T}_{\ell_i}$ are commutative. We set
\[
 \mathsf{T} := \prod_{i=1}^t \mathsf{T}_{\ell_i}^{\frac{N_1 \cdots N_t}{N_i}n_i}.
\]

\begin{proposition}\label{allcase}
 Let $L$ be a lamination on $(S,M)$. Then we have
\[
 b_T(L) \in \overline{\bigcup_{m \ge 0} C_T(\e\mathsf{T}^m(T_L))}.
\]
\end{proposition}

\begin{proof}
 By Theorem \ref{Dehn}, for $m_T \gg 0$ and $m \ge m_T$, we have
{\setlength\arraycolsep{0.5mm}
\begin{eqnarray*}
 b_T(\e\mathsf{T}^{m}(U)) &=& b_T(\e\mathsf{T}^{m_T}(U)) + \sum_{i=1}^t \Bigl(\frac{N_1 \cdots N_t}{N_i}n_i\Bigr) (m-m_T) N_i b_T(\ell_i)\\
 &=& b_T(\e\mathsf{T}^{m_T}(U)) + (m-m_T) N_1 \cdots N_t \sum_{i=1}^t n_i b_T(\ell_i)\\
 &=& b_T(\e\mathsf{T}^{m_T}(U)) + (m-m_T) N_1 \cdots N_t b_T(L_{\rm cl}).
\end{eqnarray*}}This equality gives
\[
 \lim_{m \rightarrow \infty} \frac{b_T(\e\mathsf{T}^m(U))}{m-m_T} = N_1 \cdots N_t b_T(L_{\rm cl}),
\]
 thus
\[
 b_T(L_{\rm cl}) \in \overline{\bigcup_{m \ge 0} C_T(\e\mathsf{T}^m(U))}.
\]
 Since $C_T(L_{\rm nc}) \subseteq C_T((L_{\rm nc})_{\sf pq})$ by Proposition \ref{tagpar}(1), we have
{\setlength\arraycolsep{0.5mm}
\begin{eqnarray*}
 b_T(L) = b_T(L_{\rm nc}) + b_T(L_{\rm cl}) &\in& C_T(L_{\rm nc}) + \overline{ \bigcup_{m \ge 0} C_T(\e\mathsf{T}^m(U))}\\
 &\subseteq& \overline{\bigcup_{m \ge 0} C_T((L_{\rm nc})_{\sf pq} \sqcup \e\mathsf{T}^m(U))}= \overline{\bigcup_{m \ge 0} C_T(\e\mathsf{T}^m(T_L))},
\end{eqnarray*}}
where the last equality is given by Lemma \ref{clU}(2).
\end{proof}

\begin{proof}[Proof of the first assertion of Theorem \ref{main}]
 Since $\mathsf{T}^m(T_L)$ is a tagged triangulation of $(S,M)$ for any $m \in \bZ_{\ge 0}$, Proposition \ref{allcase} finishes the proof of \eqref{vectin}. Hence the assertion holds.
\end{proof}

 To prove the second assertion of Theorem \ref{main}, we give the following results in a more general setting.

\begin{proposition}\label{sumshear}
 Let $T$ be an ideal triangulation of $(S,M)$ without self-folded triangles.
\begin{itemize}
 \item[(1)] For a laminate $\ell$, we have
\[
 \sum_{\g \in T}b_{\g,T}(\ell) \in \{0,\pm 1\}.
\]
 \item[(2)] For a tagged arc $\epsilon$ of $(S,M)$ whose both endpoints are punctures, we have
\[
 \sum_{\g \in T}b_{\g,T}(\e(\epsilon)) = \left\{
     \begin{array}{ll}
   -1 & \text{if both ends of $\epsilon$ are tagged plain},\\
   1 & \text{if both ends of $\epsilon$ are tagged notched},\\
   0 & \text{otherwise}.
     \end{array} \right.
\]
\end{itemize}
\end{proposition}

\begin{proof}
 (1) Fix a direction of $\ell$. For a closed laminate $\ell$, let $T_{\ell}$ be in \eqref{Tellc}. For a non-closed laminate $\ell$, we define a polygon $(S_{\ell},M_{\ell})$ and its ideal triangulation $T_{\ell}$ consisting of triangles $\triangle'_0,\triangle'_1,\ldots,\triangle'_r$ and $\tau_i = \triangle'_{i-1}\cap\triangle'_i$ in the same way, where we need a slight modification at the spirals. If the starting (resp., ending) end of $\ell$ is a spiral around a puncture $p$, then the triangles of $T$ incident to $p$ are $\triangle_0,\ldots,\triangle_s$ for $1<s<r$ (resp., $\triangle_t,\ldots,\triangle_r$ for $1<t<r$) as follows:
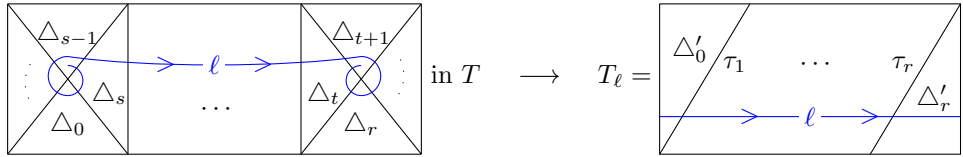
\[
\begin{tikzpicture}[baseline=-0.5mm]
 \coordinate (l) at (0,0); \coordinate (lu) at (0.8,1); \coordinate (ld) at (0.8,-1);
 \coordinate (ru) at (3.1,1); \coordinate (rd) at (3.1,-1); \coordinate (rru) at (4.7,1); \coordinate (rrd) at (4.7,-1);
 \coordinate (llu) at (-0.8,1); \coordinate (lld) at (-0.8,-1);
 \coordinate (r) at (3.9,0); \coordinate (u) at (2.1,1); \coordinate (d) at (1.8,-1);
 \node at (2,-0.4) {$\cdots$};
 \draw (l) to (lu); \draw (l) to (ld) (lld)--(l) (r)--(rru);
 \draw (lu) to (ld);
 \draw (lu) to (ru); \draw (ld) to (rd); \draw (ru) to (r); \draw (rd) to (r); \draw (ru) to (rru);
 \draw (rru) to (rrd); \draw (rrd) to (rd); \draw (rrd) to (r);
 \draw (lu) to (llu); \draw (llu) to (lld); \draw (lld) to (ld); \draw (llu) to (l);
 \draw (ru) to (rd);
 \draw[blue] (0,0.3) .. controls (1,0.15) and (2.9,0.15) ..node[fill=white,inner sep=2]{$\ell$}node[pos=0.35]{$>$}node[pos=0.65]{$>$} (3.9,0.3);
 \draw[blue] (0,0.3) arc (90:270:2.6mm);
 \draw[blue] (0,-0.22) arc (-90:90:2mm);
 \draw[blue] (3.9,0.3) arc (90:-90:2.6mm);
 \draw[blue] (3.9,-0.22) arc (-90:-270:2mm);
 \draw[loosely dotted] (4.4,-0.2) arc (-30:30:0.5);
 \draw[loosely dotted] (-0.5,0.2) arc (150:210:0.5);
 \node at(0,-0.6) {$\triangle_0$}; \node at(0,0.6) {$\triangle_{s-1}$}; \node at(0.55,-0.2) {$\triangle_s$};
 \node at(3.9,-0.6) {$\triangle_r$}; \node at(3.9,0.6) {$\triangle_{t+1}$}; \node at(3.4,-0.2) {$\triangle_t$};
\end{tikzpicture}
\ \text{in}\ T
 \hspace{4mm} \longrightarrow \hspace{4mm}
 T_{\ell}=
\begin{tikzpicture}[baseline=-0.5mm]
 \coordinate (lu) at (-2,1); \coordinate (ru) at (2,1); \coordinate (ld) at (-2,-1); \coordinate (rd) at (2,-1);
 \draw (lu)--(ru)--(rd)--(ld)--(lu);
 \draw (ld)--node[right,pos=0.6]{$\tau_1$}(-0.8,1); \draw (ru)--node[left,pos=0.4]{$\tau_r$}(0.8,-1);
 \node at(0.1,0.2) {$\cdots$}; \node at(-1.6,0.4) {$\triangle'_0$}; \node at(1.65,-0.25) {$\triangle'_r$};
 \draw[blue] (-2,-0.5)--node[fill=white,inner sep=2]{$\ell$}node[pos=0.3]{$>$}node[pos=0.7]{$>$}(2,-0.5);
\end{tikzpicture}
\]
 Let $\ell$ be an arbitrary laminate of $(S,M)$ and we consider the ideal triangulation $T_{\ell}$ consisting of triangles $\triangle'_0,\ldots,\triangle'_r$ ($\triangle'_0 = \triangle'_r$ if $\ell$ is closed).
 We call $\triangle'_i$ a {\it left} (resp., {\it right}) {\it triangle} if a side of $\triangle'_i$ is a boundary segment of $T_{\ell}$ on the left (resp., right) side of $\ell$. Then we have
\[
 b_{\tau_i,T_{\ell}}(\ell) = \left\{
     \begin{array}{ll}
   1 & \text{if $\triangle'_{i-1}$ is a left triangle and $\triangle'_{i}$ is a right triangle},\\
   -1 & \text{if $\triangle'_{i-1}$ is a right triangle and $\triangle'_{i}$ is a left triangle},\\
   0 & \text{otherwise}.
     \end{array} \right.
\]
 Therefore, we have
\begin{equation}\label{sumij}
 \sum_{k=1}^r b_{\tau_k,T_{\ell}}(\ell) = \left\{
     \begin{array}{ll}
   1 & \text{if $\triangle'_{0}$ is a left triangle and $\triangle'_{r}$ is a right triangle},\\
   -1 & \text{if $\triangle'_{0}$ is a right triangle and $\triangle'_{r}$ is a left triangle},\\
   0 & \text{otherwise}.
     \end{array} \right.
\end{equation}
 On the other hand, since $T$ has no self-folded triangles, we have
\[
 \sum_{\de \in T} b_{\de,T}(\ell) = \sum_{k=1}^r b_{\tau_k,T_{\ell}}(\ell).
\]
 Thus the assertion follows from \eqref{sumij}.

 (2) We consider $T_{\e(\epsilon)}$ as above and define the direction of $\epsilon$ from $\ell$ by the obvious way. If the starting point of $\epsilon$ is tagged plain (resp., notched), then $\triangle'_0$ is a right (resp., left) triangle. If the ending point of $\epsilon$ is tagged plain (resp., notched), then $\triangle'_0$ is a left (resp., right) triangle. Thus the assertion also follows from \eqref{sumij}.
\end{proof}

\begin{proof}[Proof of the second assertion of Theorem \ref{main}]
 Let $(S,M)$ be a closed surface with exactly one puncture $p$ and $T$ its tagged triangulation. In this case, all ends of tagged arcs of $T$ are tagged plain or they are tagged notched. Thus we can assume that $T$ is an ideal triangulation without self-folded triangles. Let $\epsilon$ and $\epsilon'$ be tagged arcs of $(S,M)$ tagged plain and notched, respectively. We only need to show that
\[
 b_T(\e(\epsilon)) \in \Bigl\{(a_{\g})_{\g \in T} \in \bR^{|T|} \mid \sum_{\g \in T} a_{\g} \le 0\Bigr\}\ \text{ and }\  b_T(\e(\epsilon')) \in \Bigl\{(a_{\g})_{\g \in T} \in \bR^{|T|} \mid \sum_{\g \in T} a_{\g} \ge 0\Bigr\},
\]
that is,
\[
 \sum_{\g \in T} b_{\g,T}(\e(\epsilon)) \le 0\ \text{ and }\ \sum_{\g \in T} b_{\g,T}(\e(\epsilon')) \ge 0.
\]
 It immediately follows from Proposition \ref{sumshear}(2).
\end{proof}

\subsection{Example of Proposition \ref{allcase}}\label{Kro}

 For an annulus $(S,M)$ with exactly two marked points, all laminates are given as follows:
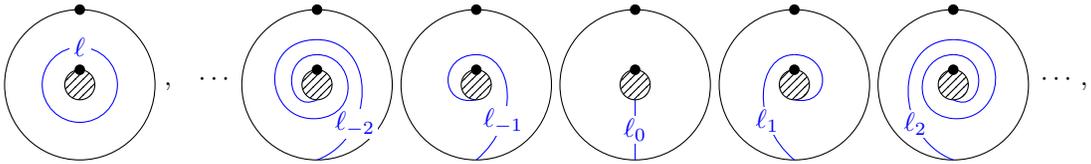
\[
\begin{tikzpicture}[baseline=0mm]
 \coordinate (0) at (0,0);
 \coordinate (u) at (0,1);
 \coordinate (cu) at (0,0.2);
 \draw[pattern=north east lines] (0) circle (2mm); \draw (0) circle (10mm);
 \draw[blue] (0) circle (5mm);
 \fill (u) circle (0.7mm); \fill (cu) circle (0.7mm);
 \node[fill=white,inner sep=2] at(0,0.5) {\color{blue}{$\ell$}};
\end{tikzpicture}\ ,
 \hspace{3mm}
 \cdots\hspace{1mm}
\begin{tikzpicture}[baseline=0mm]
 \coordinate (0) at (0,0);
 \coordinate (u) at (0,1); \coordinate (d) at (0,-1);
 \coordinate (cu) at (0,0.2); \coordinate (cd) at (0,-0.2);
 \draw[pattern=north east lines] (0) circle (2mm); \draw (0) circle (10mm);
 \draw[blue] (d) .. controls (0.8,-0.7) and (0.8,0.6) .. node[fill=white,inner sep=1,pos=0.3]{$\ell_{-2}$} (0,0.6);
 \draw[blue] (0,-0.45) .. controls (-0.7,-0.5) and (-0.8,0.6) .. (0,0.6);
 \draw[blue] (0,-0.45) .. controls (0.6,-0.4) and (0.5,0.4) .. (0,0.4);
 \draw[blue] (cd) .. controls (-0.4,-0.4) and (-0.5,0.4) .. (0,0.4);
 \fill (u) circle (0.7mm); \fill (cu) circle (0.7mm);
\end{tikzpicture}
 \hspace{1mm}
\begin{tikzpicture}[baseline=0mm]
 \coordinate (0) at (0,0);
 \coordinate (u) at (0,1); \coordinate (d) at (0,-1);
 \coordinate (cu) at (0,0.2); \coordinate (cd) at (0,-0.2);
 \draw[pattern=north east lines] (0) circle (2mm); \draw (0) circle (10mm);
 \draw[blue] (d) .. controls (0.6,-0.5) and (0.5,0.4) .. node[fill=white,inner sep=1,pos=0.3]{$\ell_{-1}$} (0,0.4);
 \draw[blue] (cd) .. controls (-0.5,-0.3) and (-0.5,0.4) .. (0,0.4);
 \fill (u) circle (0.7mm); \fill (cu) circle (0.7mm);
\end{tikzpicture}
 \hspace{1mm}
\begin{tikzpicture}[baseline=0mm]
 \coordinate (0) at (0,0);
 \coordinate (u) at (0,1); \coordinate (d) at (0,-1);
 \coordinate (cu) at (0,0.2); \coordinate (cd) at (0,-0.2);
 \draw[pattern=north east lines] (0) circle (2mm); \draw (0) circle (10mm);
 \draw[blue] (d)-- node[fill=white,inner sep=1]{$\ell_0$} (cd);
 \fill (u) circle (0.7mm); \fill (cu) circle (0.7mm);
\end{tikzpicture}
 \hspace{1mm}
\begin{tikzpicture}[baseline=0mm]
 \coordinate (0) at (0,0);
 \coordinate (u) at (0,1); \coordinate (d) at (0,-1);
 \coordinate (cu) at (0,0.2); \coordinate (cd) at (0,-0.2);
 \draw[pattern=north east lines] (0) circle (2mm); \draw (0) circle (10mm);
 \draw[blue] (d) .. controls (-0.6,-0.5) and (-0.5,0.4) .. node[fill=white,inner sep=1,pos=0.3]{$\ell_1$} (0,0.4);
 \draw[blue] (cd) .. controls (0.5,-0.3) and (0.5,0.4) .. (0,0.4);
 \fill (u) circle (0.7mm); \fill (cu) circle (0.7mm);
\end{tikzpicture}
 \hspace{1mm}
\begin{tikzpicture}[baseline=0mm]
 \coordinate (0) at (0,0);
 \coordinate (u) at (0,1); \coordinate (d) at (0,-1);
 \coordinate (cu) at (0,0.2); \coordinate (cd) at (0,-0.2);
 \draw[pattern=north east lines] (0) circle (2mm); \draw (0) circle (10mm);
 \draw[blue] (d) .. controls (-0.8,-0.7) and (-0.8,0.6) .. node[fill=white,inner sep=1,pos=0.3]{$\ell_2$} (0,0.6);
 \draw[blue] (0,-0.45) .. controls (0.7,-0.5) and (0.8,0.6) .. (0,0.6);
 \draw[blue] (0,-0.45) .. controls (-0.6,-0.4) and (-0.5,0.4) .. (0,0.4);
 \draw[blue] (cd) .. controls (0.4,-0.4) and (0.5,0.4) .. (0,0.4);
 \fill (u) circle (0.7mm); \fill (cu) circle (0.7mm);
\end{tikzpicture}
 \hspace{1mm}\cdots,
\]
where $\ell$ is closed and $\ell_m = \mathsf{T}_{\ell}^m(\ell_0)$ is elementary for $m \in \bZ$. Their shear coordinates with respect to
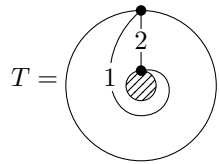
\[
 T=
\begin{tikzpicture}[baseline=0mm]
 \coordinate (0) at (0,0);
 \coordinate (u) at (0,1);
 \coordinate (cu) at (0,0.2);
 \draw[pattern=north east lines] (0) circle (2mm); \draw (0) circle (10mm);
 \draw (u)-- node[fill=white,inner sep=1]{$2$} (cu);
 \draw (u) .. controls (-0.6,0.5) and (-0.5,-0.4) .. node[fill=white,inner sep=2]{$1$} (0,-0.4);
 \draw (cu) .. controls (0.5,0.3) and (0.5,-0.4) .. (0,-0.4);
 \fill (u) circle (0.7mm); \fill (cu) circle (0.7mm);
\end{tikzpicture}
\]
 are given by\vspace{-5mm}
\[
\left.
     \begin{array}{ll}
 b_T(\ell) = (b_{1,T}(\ell),b_{2,T}(\ell)) = (1,-1),\vspace{3mm}\\
 b_T(\ell_m) =  \left\{
     \begin{array}{ll}
   (m-1,-m) & \text{if $m \ge 0$},\\
   (-m-1,m+2) & \text{if $m < 0$}.
     \end{array} \right.
     \end{array} \right.
\hspace{5mm}\begin{tikzpicture}[baseline=0mm,scale=1.2]
 \coordinate (0) at (0,0); \coordinate (x) at (1,0); \coordinate (-x) at (-1,0);
 \coordinate (y) at (0,1); \coordinate (-y) at (0,-1);
 \draw[->] (0)--(x) node[right]{$b_T(\ell_{-2})$};
 \draw (0)--(-x) node[left]{$b_T(\ell_0)$};
 \draw[->] (0)--(y) node[above]{$b_T(\ell_{-1})$};
 \draw (0)--(-y) node[below]{$b_T(\ell_1)$};
 \draw (1,-0.5)--(0)--(0.5,-1) (1,-0.75)--(0)--(0.75,-1) (1,-0.875)--(0)--(0.875,-1) (1,-0.9375)--(0)--(0.9375,-1) (1,-0.96875)--(0)--(0.96875,-1);
 \draw (0)--(1,-1) node[below right]{$b_T(\ell)$};
 \node at(1.3,-0.55){$\vdots$};
 \node at(0.8,-1.25){$\cdots$};
\end{tikzpicture}
\]
 For a tagged triangulation $T'$ of $(S,M)$, the shear coordinate cone $C_T(\e(T'))$ is given by $\{\alpha b_T(\ell_j)+\beta b_T(\ell_{j+1}) \mid \alpha, \beta \in \bZ_{\ge 0}\}$ for some $j \in \bZ$. Then the set of integer vectors which are not contained in these shear coordinate cones is $\{b_T(\{k\ell\}) \mid k \in \bZ_{>0}\}$. Taking $L=\{k\ell\}$ and $T_L=T$, Proposition \ref{allcase} means that
\[
 b_T(k\ell) = (k,-k) \in \overline{\bigcup_{m \ge 0} C_T(\e\mathsf{T}^m(T))} = \overline{\bigcup_{m \ge 0} C_T(\{\ell_m,\ell_{m+1}\})}.
\]
It is described in the above picture.

\section{Cluster algebras}\label{pfgfan}

\subsection{Cluster algebras and triangulated surfaces}\label{clalg}

 We briefly recall cluster algebras with principal coefficients \cite{FZ07}. For that, we need to prepare some notations. Let $n \in \bZ_{\ge 0}$ and $\cF:={\mathbb Q}(t_1,\ldots,t_{2n})$ be the field of rational functions in $2n$ variables over ${\mathbb Q}$.

\begin{definition}
 $(1)$ A {\it seed with coefficients} is a pair $(\x,Q)$ consisting of the following data:
\begin{itemize}
 \item[(a)] $\x=(x_1,\ldots,x_n,y_1,\ldots,y_n)$ is a free generating set of $\cF$ over $\mathbb{Q}$.
 \item[(b)] $Q$ is a quiver without loops and $2$-cycles whose vertices are $\{1,\ldots,2n\}$.
\end{itemize}
 Then we refer to $\x$ as the {\it cluster}, to each $x_i$ as a {\it cluster variable} and $y_i$ as a {\it coefficient}.

 $(2)$ For a seed $(\x,Q)$ with coefficients, the {\it mutation $\mu_k(\x,Q)=(\x',Q')$ in direction $k$} $(1 \le k \le n)$ is defined as follows:
\begin{itemize}
 \item[(a)] $\x'=(x'_1,\ldots,x'_n,y_1,\ldots,y_n)$ is defined by
 \begin{equation*}
  x_k x'_k = \prod_{(j \rightarrow k)\text{ in }Q}x_j\prod_{(j \rightarrow k)\text{ in }Q}y_{j-n}+\prod_{(j \leftarrow k)\text{ in }Q}x_j\prod_{(j \leftarrow k)\text{ in }Q}y_{j-n} \ \ \text{and} \ \ x'_i = x_i \ \ \text{if} \ \ i \neq k,
 \end{equation*}
where $x_{n+1}=\cdots=x_{2n}=1=y_{1-n}=\cdots=y_0$.
 \item[(b)] $Q'$ is the quiver obtained from $Q$ by the following steps:\par
 \begin{itemize}
  \item[(i)] For any path $i \rightarrow k \rightarrow j$, add an arrow $i \rightarrow j$.
  \item[(ii)] Reverse all arrows incident to $k$.
  \item[(iii)] Remove a maximal set of disjoint $2$-cycles.
 \end{itemize}
\end{itemize}
\end{definition}

We remark that $\mu_k$ is an involution, that is, we have $\mu_k\mu_k(\x,Q)=(\x,Q)$. Moreover, it is elementary that $\mu_k(\x,Q)$ is also a seed with coefficients.

 For a quiver $Q$ without loops and $2$-cycles whose vertices are $\{1,\ldots,n\}$. The {\it framed quiver} associated with $Q$ is the quiver $\hat{Q}$ obtained from $Q$ by adding vertices $\{1',\ldots,n'\}$ and arrows $\{i \rightarrow i' \mid 1 \le i \le n\}$. We fix a seed $(\x=(x_1,\ldots,x_n,y_1,\ldots,y_n),\hat{Q})$ with coefficients, called the {\it initial seed}. We also call each $x_i$ the {\it initial cluster variable}.

\begin{definition}
 The {\it cluster algebra} $\cA(Q)=\cA(\x,\hat{Q})$ {\it with principal coefficients} for the initial seed $(\x,\hat{Q})$ is a $\bZ$-subalgebra of $\cF$ generated by the cluster variables and the coefficients obtained by all sequences of mutations from $(\x,\hat{Q})$.
\end{definition}

 One of the remarkable properties of cluster algebras with principal coefficients is the strongly Laurent phenomenon \cite[Proposition 3.6]{FZ07}, that is, $\cA(Q) \subseteq \bZ[x_1^{\pm 1},\ldots,x_n^{\pm 1},y_1,\ldots,y_n]$. We consider the $\bZ^n$-grading in $\bZ[x_1^{\pm 1},\ldots,x_n^{\pm 1},y_1,\ldots,y_n]$ given by
\[
 \deg(x_i)=e_i,\ \ \deg(y_j)=(\#\{i \rightarrow j\text{ in }Q\}-\#\{i \leftarrow j\text{ in }Q\})_{1 \le i \le n},
\]
where $e_1,\ldots,e_n$ are the standard basis vectors in $\bZ^n$. Every cluster variable $x$ of $\cA(Q)$ is homogeneous with respect to the $\bZ^n$-grading, and its degree $g_Q(x)$ is called {\it $g$-vector} of $x$ \cite[Proposition 6.1]{FZ07}. We denote by $\clus Q$ the set of clusters in $\cA(Q)$ and by $\clv Q$ the set of cluster variables in $\cA(Q)$.

 Let $T$ be a tagged triangulation of $(S,M)$. Fomin, Shapiro and Thurston \cite{FST} constructed a quiver $Q_T$ without loops and $2$-cycles as follows: Any tagged triangulation is obtained by gluing together a number of puzzle pieces in Table \ref{QT} and by simultaneous changing all tags at some punctures (see \cite[Remark 4.2]{FST} for details). The vertices of $Q_T$ are arcs of $T$ and its arrows are obtained as in Table \ref{QT} for puzzle pieces of $T$, where we remove arrows incident to $\partial S$.
\begin{table}[ht]
\centering
\begin{tabular}{c|c|c|c|c}
\begin{tikzpicture}[baseline=-9mm]
 \node at(0,0.3) {Puzzle}; \node at(0,-0.3) {pieces};
\end{tikzpicture}
   &
\begin{tikzpicture}[baseline=-3mm]
 \coordinate (0) at (0,0);
 \coordinate (1) at (120:1.5);
 \coordinate (2) at (180:1.5);
 \draw (0) to node[right]{$\de_2$} (1);
 \draw (1) to node[left]{$\de_1$} (2);
 \draw (0) to node[below]{$\de_3$} (2);
 \fill(0) circle (0.7mm); \fill (1) circle (0.7mm); \fill (2) circle (0.7mm);
\end{tikzpicture}
   &
\begin{tikzpicture}[baseline=-20mm]
 \coordinate (0) at (0,0);
 \coordinate (1) at (0,-1);
 \coordinate (2) at (0,-2);
 \draw (2) to [out=180,in=180] node[left]{$\de_1$} (0);
 \draw (2) to [out=0,in=0] node[right]{$\de_2$} (0);
 \draw (1) to node[left=-3]{$\de_3$} (2);
 \draw (1) to [out=60,in=120,relative] node[pos=0.2]{\rotatebox{40}{\footnotesize $\bowtie$}} node[fill=white,inner sep=0.8]{$\de_4$} (2);
 \fill(0) circle (0.7mm); \fill (1) circle (0.7mm); \fill (2) circle (0.7mm);
\end{tikzpicture}
   &
\begin{tikzpicture}[baseline=-22mm,scale=1.2]
 \coordinate (0) at (0,0); \node at (0,-0.3) {$\de_1$};
 \coordinate (1) at (-0.5,-0.8); \fill (1) circle (0.7mm);
 \coordinate (1') at (0.5,-0.8); \fill (1') circle (0.7mm);
 \coordinate (2) at (0,-2); \fill (2) circle (0.7mm);
 \draw (0,-1) circle (1);
 \draw (1) to node[fill=white,inner sep=1]{$\de_2$} (2);
 \draw (1) to [out=-60,in=-120,relative] node[pos=0.2]{\rotatebox{170}{\footnotesize $\bowtie$}} node[fill=white,inner sep=1]{$\de_3$} (2);
 \draw (1') to node[fill=white,inner sep=1]{$\de_4$} (2);
 \draw (1') to [out=60,in=120,relative] node[pos=0.2]{\rotatebox{180}{\footnotesize $\bowtie$}} node[fill=white,inner sep=1]{$\de_5$} (2);
\end{tikzpicture}
   &
\begin{tikzpicture}[baseline=-8mm,scale=1.1]
 \coordinate (0) at (0,0);
 \coordinate (u) at (90:1);
 \coordinate (r) at (-30:1);
 \coordinate (l) at (210:1);
 \draw (0) to node[left]{$\de_1$} (u);
 \draw (0) to [out=-60,in=-120,relative] node[pos=0.8]{\rotatebox{20}{\footnotesize $\bowtie$}} node[right]{$\de_1'$} (u);
 \draw (0) to node[right=5,above=-3]{$\de_2$} (r);
 \draw (0) to [out=-60,in=-120,relative] node[pos=0.8]{\rotatebox{100}{\footnotesize $\bowtie$}} node[below]{$\de_2'$} (r);
 \draw (0) to node[below=7,right=-5]{$\de_3$} (l);
 \draw (0) to [out=-60,in=-120,relative] node[pos=0.8]{\rotatebox{-30}{\footnotesize $\bowtie$}} node[left=2,above=-1]{$\de_3'$} (l);
 \fill (0) circle (0.7mm); \fill (u) circle (0.7mm); \fill (l) circle (0.7mm); \fill (r) circle (0.7mm);
\end{tikzpicture}
\\\hline
\begin{tikzpicture}[baseline=-5mm]
 \node at(0,0.3) {Corresponding}; \node at(0,-0.3) {quivers};
\end{tikzpicture}
   &
\begin{tikzpicture}[baseline=2mm]
 \node (3) at (0,0) {$\de_3$};
 \node (2) at (60:1.5) {$\de_2$};
 \node (1) at (120:1.5) {$\de_1$};
 \draw[->] (3) -- (2); \draw[->] (2) -- (1); \draw[->] (1) -- (3);
\end{tikzpicture}
   &
\begin{tikzpicture}[baseline=0mm]
 \node (3) at (0,0.5) {$\de_3$};
 \node (4) at (0,-0.3) {$\de_4$};
 \node (2) at (60:1.5) {$\de_2$};
 \node (1) at (120:1.5) {$\de_1$};
 \draw[->] (3) -- (2); \draw[->] (2) -- (1); \draw[->] (1) -- (3);
 \draw[->] (4) -- (2); \draw[->] (1) -- (4);
\end{tikzpicture}
   &
\begin{tikzpicture}[baseline=-14mm,scale=1.2]
 \node (1) at (0,0) {$\de_1$};
 \node (2) at (-0.7,-0.7) {$\de_2$};
 \node (3) at (-0.5,-1.5) {$\de_3$};
 \node (4) at (0.7,-0.7) {$\de_4$};
 \node (5) at (0.5,-1.5) {$\de_5$};
 \draw[->] (1) -- (2); \draw[->] (2) -- (4); \draw[->] (4) -- (1);
 \draw[->] (1) -- (3); \draw[->] (3) -- (4);
 \draw[->] (2) -- (5); \draw[->] (5) -- (1); \draw[->] (3) -- (5);
\end{tikzpicture}
   &
\begin{tikzpicture}[baseline=-3mm]
 \node (1) at (90:0.5) {$\de_1$};
 \node (2) at (90:1.2) {$\de_1'$};
 \node (3) at (-30:0.5) {$\de_2$};
 \node (4) at (-30:1.2) {$\de_2'$};
 \node (5) at (210:0.5) {$\de_3$};
 \node (6) at (210:1.2) {$\de_3'$};
 \draw[->] (1) -- (3); \draw[->] (3) -- (5); \draw[->] (5) -- (1);
 \draw[->] (2) to [out=40,in=140,relative] (4); \draw[->] (4) to [out=40,in=140,relative] (6); \draw[->] (6) to [out=40,in=140,relative] (2);
 \draw[->] (2) to [out=40,in=140,relative] (3); \draw[->] (4) to [out=40,in=140,relative] (5); \draw[->] (6) to [out=40,in=140,relative] (1);
 \draw[->] (1) to [out=40,in=140,relative] (4); \draw[->] (3) to [out=40,in=140,relative] (6); \draw[->] (5) to [out=40,in=140,relative] (2);
\end{tikzpicture}
\end{tabular}\vspace{3mm}
\caption{Puzzle pieces and the corresponding quivers}\label{QT}
\end{table}
 Thus we have the cluster algebra $\cA(Q_T)$ associated with $T$.

 We denote by $\bT_T$ the set of tagged triangulations of $(S,M)$ obtained from $T$ by sequences of flips, and by $\bA_T$ the set of tagged arcs of each tagged triangulation contained in $\bT_T$. Cluster algebras defined from triangulated surfaces have the following properties.

\begin{theorem}\label{bijtc}
 Let $T$ be a tagged triangulation of $(S,M)$.
\begin{itemize}
 \item[(1)] \cite[Theorem 7.11]{FST}\cite[Theorem 6.1]{FoT} There is a bijection
\[
 x_{(-)} : \bA_T \longleftrightarrow \clv Q_T.
\]
 Moreover, it induces a bijection
\[
 x_{(-)} : \bT_T \longleftrightarrow \clus Q_T
\]
which sends $T$ to the initial cluster in $\cA(Q_T)$ and commutes with flips and mutations.
 \item[(2)] \cite[Theorem 10.0.5]{Lab10}\cite[Theorem 7.1]{Lab09b}\cite[Proposition 5.2]{Re14b} For each $\de \in \bA_T$, we have
\[
 -b_T(\e(\de)) = g_{Q_T}(x_{\de}).
\]
\end{itemize}
\end{theorem}

 Note that, in another way, Theorem \ref{bijtc}(2) can be directly given by the cluster expansion formula in \cite{Y18b}. Moreover, it was proved in \cite[Theorem 8.6]{FeT} for orbifolds in the same way as \cite[Proposition 5.2]{Re14b}.

\subsection{Proof of Theorem \ref{gfan}}

 We recall the following notion to prove Theorem \ref{gfan}.

\begin{definition}\cite{BQ}
 Let $(S,M)$ be an arbitrary marked surface. The {\it tagged rotation} of a tagged arc $\de$ of $(S,M)$ is the tagged arc $\rho(\de)$ defined as follows:
\begin{itemize}
 \item If $\de$ has an endpoint $o$ on a component $C$ of $\partial S$, then $\rho(\de)$ is obtained from $\de$ by moving $o$ to the next marked point on $C$ in the counterclockwise direction;
 \item If $\de$ has an endpoint at a puncture $p$, then $\rho(\de)$ is obtained from $\de$ by changing its tags at $p$.
\end{itemize}
\end{definition}

 By Theorem \ref{flip}, we have
\begin{equation}\label{bT}
 \bT =  \left\{
     \begin{array}{ll}
   \bT_T \sqcup \bT_{\rho T} & \text{if $(S,M)$ is a closed surface with exactly one puncture},\\
   \bT_T & \text{otherwise}.
     \end{array} \right.
\end{equation}

 Let $S^{\ast}$ be the same surface as $S$ oriented in the opposite direction and $M^{\ast}=M$. For a tagged arc or laminate $\g$ of $(S,M)$, we denote by $\g^{\ast}$ the corresponding one of $(S^{\ast},M^{\ast})$. In particular, the tagged triangulation $T^{\ast}$ of $(S^{\ast},M^{\ast})$ is naturally induced by $T \in \bT$ and we have $Q_{T^{\ast}}=Q_T^{\rm op}$. By Theorem \ref{bijtc}(1), the composition of maps $\rho^{-1}(-)$, $(-)^{\ast}$ and $x_{(-)}$ gives a bijection
\[
 x_{(\rho^{-1}(-))^{\ast}} : \bA_{\rho T} \longleftrightarrow \clv Q_T^{\rm op}.
\]
 Moreover, it induces a bijection
\[
 x_{(\rho^{-1}(-))^{\ast}} : \bT_{\rho T} \longleftrightarrow \clus Q_T^{\rm op}
\]
 which sends $\rho T$ to the initial cluster in $\cA(Q_T^{\rm op})$ and commutes with flips and mutations.

\begin{theorem}\label{Qop}
 Let $T$ be a tagged triangulation of $(S,M)$. Then for each $\de \in \bA_{\rho T}$, we have
\[
 b_T(\e(\de))=g_{Q_T^{\rm op}}(x_{(\rho^{-1}(\de))^{\ast}}).
\]
\end{theorem}

\begin{proof}
 For a tagged arc $\de$ of $(S,M)$, the equalities
\[
 b_T(\e(\rho(\de))) = b_T((\e(\de^{\ast}))^{\ast}) = -b_{T^{\ast}}(\e(\de^{\ast}))
\]
hold. Since $Q_{T^{\ast}}=Q_T^{\rm op}$, Theorem \ref{bijtc}(2) gives
\[
 -b_{T^{\ast}}(\e(\de^{\ast})) = g_{Q_{T^{\ast}}}(x_{\de^{\ast}}) = g_{Q_T^{\rm op}}(x_{\de^{\ast}})
\]
for $\de \in \bA_T$, hence $b_T(\e(\rho(\de)))=g_{Q_T^{\rm op}}(x_{\de^{\ast}})$.
\end{proof}

\begin{proof}[Proof of Theorem \ref{gfan}]
 By Theorems \ref{bijtc} and \ref{Qop}, we have
\[
 \bigcup_{T' \in \bT_T}C_T(\e(T')) = \bigcup_{\x \in \clus Q_T}\bigl(-C_{Q_T}(\x)\bigr)\ \ \text{and}\ \ \bigcup_{T' \in \bT_{\rho T}}C_T(\e(T')) = \bigcup_{\x \in \clus Q_T^{\rm op}}C_{Q_T^{\rm op}}(\x).
\]
 If $(S,M)$ is a closed surface with exactly one puncture, then $\bT_T$ and $\bT_{\rho T}$ coincide with $\bT^+$ and $\bT^-$ in Theorem \ref{main}, respectively. Therefore, the assertion follows from Theorem \ref{main} and \eqref{bT}.
\end{proof}

\subsection{Example for a cluster algebra}\label{excl}

 For the tagged triangulation $T$ in Subsection \ref{Kro}, the quiver $Q_T$ is the Kronecker quiver $1 \leftleftarrows 2$. The set $\clus Q_T$ is described by
\[
\xymatrix@C=7mm@R=5mm{
 (x_1,x_2) \ar@{-}[r] \ar@{-}[d] & \biggl(\cfrac{x_2^2+y_1}{x_1},x_2\biggr) \ar@{-}[r] & \biggl(\cfrac{x_2^2+y_1}{x_1},\cfrac{x_2^4+y_1^2y_2x_1^2+2y_1x_2^2+y_1^2}{x_1^2x_2}\biggr) \ar@{-}[r] & \cdots
\\
 \biggl(x_1,\cfrac{y_2x_1^2+1}{x_2}\biggr) \ar@{-}[r] & \biggl(x',\cfrac{y_2x_1^2+1}{x_2}\biggr) \ar@{-}[r] & \biggl(x',x''\biggr) \ar@{-}[r] & \cdots
}
\]
where
\[
 x' = \cfrac{y_1y_2^2x_1^4+2y_1y_2x_1^2+x_2^2+y_1}{x_1x_2^2},\ x'' = \cfrac{y_1^2y_2^3x_1^6+3y_1^2y_2^2x_1^4+2y_1y_2x_1^2x_2^2+x_2^4+3y_1^2y_2x_1^2+2y_1x_2^2+y_1^2}{x_1^2x_2^3}.
\]
 The corresponding $g$-vectors are as follows:
\[
\xymatrix@C=7mm@R=5mm{
 (1,0), (0,1) \ar@{-}[r] \ar@{-}[d] & (-1,2), (0,1) \ar@{-}[r] & (-1,2), (-2,3) \ar@{-}[r] & \cdots
\\
 (1,0), (0,-1) \ar@{-}[r] & (-1,0), (0,-1) \ar@{-}[r] & (-1,0), (-2,1) \ar@{-}[r] & \cdots
}
\]
 The $g$-vector cones of clusters are reflections of the corresponding shear coordinate cones in Subsection \ref{Kro} as follows:
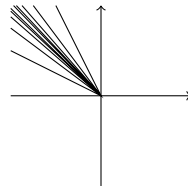
\[
\begin{tikzpicture}[baseline=0mm,scale=1.2]
 \coordinate (0) at (0,0); \coordinate (x) at (1,0); \coordinate (-x) at (-1,0);
 \coordinate (y) at (0,1); \coordinate (-y) at (0,-1);
 \draw[->] (0)--(x); \draw (0)--(-x); \draw[->] (0)--(y); \draw (0)--(-y);
 \draw (-1,0.5)--(0)--(-0.5,1) (-1,0.75)--(0)--(-0.75,1) (-1,0.875)--(0)--(-0.875,1) (-1,0.9375)--(0)--(-0.9375,1) (-1,0.96875)--(0)--(-0.96875,1);
\end{tikzpicture}
\]

\section{Representation theory}\label{Rep}

\subsection{$\tau$-tilting theory and cluster tilting theory}

 In this subsection, we recall $\tau$-tilting and cluster tilting theory to prepare for the proofs of Theorem \ref{gtame} and Corollary \ref{connect}.

 First, we recall $\tau$-tilting theory \cite{AIR}. Let $\Lambda$ be a finite dimensional algebra over a field. We denote by $\Mod \Lambda$ (resp., $\proj \Lambda$) the category of finitely generated (resp., finitely generated projective) left $\Lambda$-modules. We denote by $\tau$ the Auslander-Reiten translation of $\Mod \Lambda$ and by $|M|$ is the number of non-isomorphic indecomposable direct summands of $M \in \Mod \Lambda$. Let $M \in \Mod \Lambda$ and $P \in \proj \Lambda$. We say that a pair $(M,P)$ is

\begin{itemize}
 \item {\it $\tau$-rigid} if $\Hom_{\Lambda}(M,\tau M)=0=\Hom_{\Lambda}(P,M)$;
 \item {\it $\tau$-tilting} if $(M,P)$ is $\tau$-rigid and $|\Lambda|=|M|+|P|$
 \item {\it basic} if $M$ and $P$ are basic;
 \item a {\it direct summand} of $(M',P') \in \Mod \Lambda \times \proj \Lambda$ if $M$ is a direct summand of $M'$ and $P$ is a direct summand of $P'$;
 \item {\it indecomposable} if $(M,P)$ is basic and $|M|+|P|=1$.
\end{itemize}

\noindent Recall that we denote by $\sttilt \Lambda$ the set of isomorphism classes of basic $\tau$-tilting pairs in $\Mod \Lambda$. For $N \in \sttilt \Lambda$ and an indecomposable direct summand $N'$ of $N$, there is a unique indecomposable $\tau$-rigid pair $N''$ such that $N/N' \oplus N'' \in \sttilt\Lambda$ \cite[Theorem 0.4]{AIR}. Therefore, one can define mutations in $\sttilt\Lambda$.

 Let $\Lambda = \bigoplus_{i=1}^{n}P_i$ be a decomposition of $\Lambda$, where $P_i$ is an indecomposable projective $\Lambda$-module. Then $[P_1],\ldots,[P_n]$ form a basis for $K_0(\proj \Lambda)$, thus there is a natural bijection between $K_0(\proj \Lambda)$ and $\bZ^n$. Let $M \in \Mod \Lambda$. There is a minimal projective presentation of $M$
\[
 P^1 \rightarrow P^0 \rightarrow M \rightarrow 0.
\]
 We set
\[
 g_{\Lambda}(M) := [P^0]-[P^1] \in K_0(\proj \Lambda) \simeq \bZ^n,
\]
called the {\it $g$-vector of $M$}. We denote by $\trigid\Lambda$ the set of isomorphism classes of indecomposable $\tau$-rigid pairs in $\Mod \Lambda$. The {\it $g$-vector of $(M,P) \in \trigid\Lambda$} is $g_{\Lambda}(M,P):=g_{\Lambda}(M)-g_{\Lambda}(P)$.

 For our aim, we also need to consider the opposite algebra $\Lambda^{\rm op}$ of $\Lambda$. For $M \in \Mod \Lambda$, the notation $\Tr M$ denotes the transpose of $M$. We define $(-)^{\ast}:=\Hom_{\Lambda}(-,\Lambda) : \proj\Lambda \longleftrightarrow \proj\Lambda^{\rm op}$. Then $(-)^{\ast}$ gives $K_0(\proj \Lambda) \simeq K_0(\proj \Lambda^{\rm op})$.

\begin{theorem}\cite[Theorem 2.14]{AIR}\cite[Subsection 3.4]{F}\label{varphi}
 There is a bijection
\[
 \varphi : \trigid\Lambda \longleftrightarrow \trigid\Lambda^{\rm op}
\]
given by $(M,P) \mapsto (\Tr M \oplus P^{\ast},M_{\rm pr})$ such that
\[
 g_{\Lambda}(M,P)=-g_{\Lambda^{\rm op}}(\Tr M \oplus P^{\ast},M_{\rm pr}),
\]
 where $M_{\rm pr}$ is a maximal projective direct summand of $M$. The map $\varphi$ induces a bijection
\[
 \varphi : \sttilt\Lambda \longleftrightarrow \sttilt\Lambda^{\rm op}
\]
which sends $(\Lambda,0)$ to $(0,\Lambda)$ and commutes with mutations.
\end{theorem}

 Next, we recall cluster tilting theory in $2$-Calabi-Yau triangulated categories. Let $\cC$ be a Hom-finite Krull-Schmidt $2$-Calabi-Yau triangulated category. We call $X \in \cC$ {\it rigid} if $\Hom_{\cC}(X,X[1])=0$. We denote by $\add U$ the category of all direct summands of finite direct sums of copies of $U$. We call $U \in \cC$ {\it cluster tilting} if $\add U = \{X \in \cC \mid \Hom_{\cC}(U,X[1])=0\}$. We denote by $\rigid\cC$ the set of isomorphism classes of indecomposable rigid objects in $\cC$. Recall that we denote by $\ctilt\cC$ the set of isomorphism classes of basic cluster tilting objects in $\cC$. We assume that $\cC$ has cluster tilting objects, that is, $\ctilt\cC \neq \emptyset$. In this case, any maximal rigid object in $\cC$ is cluster tilting \cite[Theorem 2.6]{ZZ}. Iyama and Yoshino \cite{IY} gave mutations in $\ctilt\cC$ (see also \cite{BMRRT}).

 Let $U = \bigoplus_{i=1}^{n}U_i$ be a decomposition of $U$, where $U_i$ is indecomposable. Then $[U_1],\ldots,[U_n]$ form a basis for $K_0(\add U)$, thus there is a natural bijection between $K_0(\add U)$ and $\bZ^n$. For $U \in \ctilt\cC$ and $X \in \cC$, there is a triangle
\[
 U^1 \rightarrow U^0 \rightarrow X \rightarrow U_1[1],
\]
where $U^1, U^0 \in \add U$. We define
\[
 g_{U}(X) := [U^0]-[U^1] \in K_0(\add U) \simeq \bZ^n,
\]
called the {\it $g$-vector of $X$ with respect to $U$}.

 There is a close relationship between cluster tilting theory and $\tau$-tilting theory as follows.

\begin{theorem}\cite[Theorem 4.1]{AIR}\label{sfE}
 Let $U \in \ctilt\cC$ and $\Lambda = \End_{\cC}(U)^{\rm op}$. Then there is a bijection
\[
 \sH := \Hom_{\cC}(U,-) : \rigid\cC \longleftrightarrow \trigid \Lambda
\]
 such that
\[
 g_{U}(X)=g_{\Lambda}(\sH(X))
\]
 for $X \in \rigid\cC$. Moreover, it induces a bijection
\[
 \sH : \ctilt\cC \longleftrightarrow \sttilt \Lambda
\]
which sends $U$ to $(\Lambda,0)$ and commutes with mutations.
\end{theorem}

 For $\bullet \in \{+,-\}$, we denote by $\rigid^{\bullet}\cC$ (resp., $\trigid^{\bullet}\Lambda$) the set of indecomposable direct summands of an object in $\ctilt^{\bullet}\cC$ (resp., $\sttilt^{\bullet}\Lambda$). Clearly, the map $\sH$ in Theorem \ref{sfE} gives bijections
\[
 \rigid^{\bullet}\cC \longleftrightarrow \trigid^{\bullet}\Lambda\ \text{ and }\  \ctilt^{\bullet}\cC \longleftrightarrow \sttilt^{\bullet}\Lambda.
\]

\subsection{Representation theory and cluster algebras}

 We consider a relationship between representation theory and cluster algebras to prove Theorem \ref{gtame} and Corollary \ref{connect}. For a quiver with potential $(Q,W)$, we have the associated Jacobian algebra $J(Q,W)$, Ginzburg differential graded algebra $\Gamma_{Q,W}$, and generalized cluster category $\cC_{Q,W}$ (see e.g. \cite{A,DWZ,G,K08,K11} for details). The following is the main result in the additive categorification of cluster algebras.

\begin{theorem}\label{CC}
 Let $Q$ be a quiver without loops and $2$-cycles and $W$ a non-degenerate potential of $Q$ such that $J(Q,W)$ is finite dimensional.
\begin{itemize}
 \item[(1)] \cite[Theorem 2.1]{A} The category $\cC_{Q,W}$ is a Hom-finite Krull-Schmidt $2$-Calabi-Yau triangulated category with a cluster tilting object $\Gamma_{Q,W}$.
 \item[(2)]\cite[Theorem 6.3]{FK}\cite[Corollary 3.5]{CKLP}
 There is a bijection
\[
 \X : \rigid^+\cC_{Q,W} \longleftrightarrow \clv Q
\]
 such that
\[
 g_{\Gamma_{Q,W}}(X)=g_Q(\X(X))
\]
 for $X \in \rigid^+\cC_{Q,W}$. Moreover, it induces a bijection
\[
 \X : \ctilt^+\cC_{Q,W} \longleftrightarrow \clus Q
\]
which sends $\Gamma_{Q,W}$ to the initial cluster in $\cA(Q)$ and commutes with mutations.
\end{itemize}
\end{theorem}

 Note that the map $\X$ in Theorem \ref{CC}(2) is called the {\it cluster character} associated with $(Q,W)$ (see e.g. \cite{BY,CC,Pa,Pl11a,Pl11b}).

 We also study $\rigid^-\cC_{Q,W}$ and $\ctilt^-\cC_{Q,W}$. We have
\[
\End_{\cC_{Q,W}}(\Gamma_{Q,W})^{\rm op} \simeq J(Q,W)\ \text{and}\ J(Q,W)^{\rm op} \simeq J(Q^{\rm op},W^{\rm op}),
\]
 where $Q^{\rm op}$ is the opposite quiver of $Q$ and $W^{\rm op}$ is a non-degenerate potential of $Q^{\rm op}$ corresponding to $W$.

\begin{corollary}\label{CC'}
 Let $Q$ be a quiver without loops and $2$-cycles and $W$ a non-degenerate potential of $Q$ such that $J(Q,W)$ is finite dimensional. Then there is a bijection
\[
 \X' : \rigid^-\cC_{Q,W} \longleftrightarrow \clv Q^{\rm op}
\]
 such that
\[
 g_{\Gamma_{Q,W}}(X)=-g_{Q^{\rm op}}(\X'(X))
\]
 for $X \in \rigid^-\cC_{Q,W}$. Moreover, it induces a bijection
\[
 \X' : \ctilt^-\cC_{Q,W} \longleftrightarrow \clus Q^{\rm op}
\]
which sends $\Gamma_{Q,W}[1]$ to the initial cluster in $\cA(Q^{\rm op})$ and commutes with mutations.
\end{corollary}

\begin{proof}
 Let $\X'$ be the following composition:
{\setlength\arraycolsep{0.5mm}
\begin{eqnarray*}
 \rigid^-\cC_{Q,W} &\overset\sH{\longrightarrow}& \trigid^-\End_{\cC_{Q,W}}(\Gamma_{Q,W})^{\rm op} \longleftrightarrow \trigid^-J(Q,W)\\
 &\overset{\varphi}{\longrightarrow}& \trigid^+J(Q,W)^{\rm op} \longleftrightarrow \trigid^+J(Q^{\rm op},W^{\rm op})\\
 &\overset{\sH^{-1}}{\longrightarrow}& \rigid^+\cC_{Q^{\rm op},W^{\rm op}}\\
 &\overset{\X}{\longrightarrow}& \clv Q^{\rm op}.
\end{eqnarray*}}By Theorems \ref{varphi}, \ref{sfE} and \ref{CC}, it induces a bijection between $\ctilt^-\cC_{Q,W}$ and $\clus Q^{\rm op}$ which sends $\Gamma_{Q,W}[1]$ to the initial cluster in $\cA(Q^{\rm op})$ and commutes with mutations. Moreover, we have the equalities
\[
 g_{\Gamma_{Q,W}}(X) = g_{J(Q,W)}(\sH(X)) = -g_{J(Q,W)^{\rm op}}(\varphi\sH(X)) = -g_{\Gamma_{Q^{\rm op},W^{\rm op}}}(\sH^{-1}\varphi\sH(X)) = -g_{Q^{\rm op}}(\X'(X))
\]
 for $X \in \rigid^-\cC_{Q,W}$.
\end{proof}

 For a tagged triangulation $T$ of $(S,M)$, we consider a non-degenerate potential $W$ of $Q_T$ such that $J(Q_T,W)$ is finite dimensional. It is known that such a potential $W$ exists.

\begin{proposition}\label{existW}
 Let $T$ be a tagged triangulation of $(S,M)$. Then there is a non-degenerate potential $W$ of $Q_T$ such that $J(Q_T,W)$ is finite dimensional.
\end{proposition}

\begin{proof}
 For a sphere $(S,M)$ with exactly four punctures, such a potential $W$ was given in \cite{GG} (see also \cite{GLS}). Suppose that $(S,M)$ is not a sphere with exactly four punctures. Labardini-Fragoso \cite{Lab09a,Lab16} defined a potential $W$ of $Q_T$ for any tagged triangulation $T$ of $(S,M)$, and showed that it is non-degenerate except for a sphere with exactly five punctures, which in this case was proved in \cite{GLS}. Finite dimensionally of $J(Q_T,W)$ was proved in \cite{Lab09a} for $(S,M)$ with non-empty boundary and in \cite{Lad12} for $(S,M)$ with empty boundary, where it was proved independently in \cite{TV} for spheres.
\end{proof}

\subsection{Proofs of Theorem \ref{gtame} and Corollary \ref{connect}}

 We keep the notations in the previous subsection. Let $\Gamma, U=\bigoplus_{i=1}^n U_i \in \ctilt \cC$ and $N=\bigoplus_{i=1}^n N_i \in\sttilt\Lambda$, where $U_i$ and $N_i$ are indecomposable. We define {\it $g$-vector cones}
\[
 C_{\Gamma}(U) := \biggl\{\sum_{i=1}^{n}a_ig_{\Gamma}(U_i) \mid a_i \in \bR_{\ge0}\biggr\}\ \ \text{and}\ \  C_{\Lambda}(N) := \biggl\{\sum_{i=1}^{n}a_ig_{\Lambda}(N_i) \mid a_i \in \bR_{\ge0}\biggr\}.
\]

\begin{proof}[Proof of Theorem \ref{gtame}]
 Let $T$ be a tagged triangulation of $(S,M)$ and $W$ a non-degenerate potential of $Q=Q_T$ such that $J(Q,W)$ is finite dimensional. By Theorem \ref{sfE}, we have
\[
 C_{\Gamma_{Q,W}}(U) = C_{J(Q,W)}(\sH(U))
\]
 for $U \in \ctilt \cC_{Q,W}$. Therefore, we only need to prove the assertion for $\cC_{Q,W}$. By Theorem \ref{CC}(2) and Corollary \ref{CC'}, the equalities
\[
 \bigcup_{U \in \ctilt^+ \cC_{Q,W}}C_{\Gamma_{Q,W}}(U) = \bigcup_{\x \in \clus Q_T}C_{Q_T}(\x)\ \ \text{and}\ \ \bigcup_{U \in \ctilt^- \cC_{Q,W}}C_{\Gamma_{Q,W}}(U) = \bigcup_{\x \in \clus Q_T^{\rm op}}\bigl(-C_{Q_T^{\rm op}}(\x)\bigr).
\]
 hold. Thus the assertion follows from Theorem \ref{gfan}.
\end{proof}

\begin{proof}[Proof of Corollary \ref{connect}]
 A $g$-vector cone $C_{\Gamma_{Q,W}}(U)$ has dimension $n$ for any $U \in \ctilt\cC_{Q,W}$ \cite[Theorem 2.4]{DK}. For $U \ncong V \in \ctilt\cC_{Q,W}$, $C_{\Gamma_{Q,W}}(U)$ and $C_{\Gamma_{Q,W}}(V)$ have no intersections except for their boundaries \cite[Corollary 6.7]{DIJ}. Thus there are no cluster tilting objects in $\ctilt\cC_{Q,W} \setminus \ctilt^{\pm}\cC_{Q,W}$ by Theorem \ref{gtame}. The assertion follows from Theorem \ref{flip}.
\end{proof}

\subsection{Example for representation theory}

 For the tagged triangulation $T$ in Subsection \ref{Kro}, the quiver $Q_T$ is the Kronecker quiver $1 \leftleftarrows 2$. The set $\sttilt J(Q_T,0)$ is as follows:
\[
\xymatrix@C=7mm@R=5mm{
 \Bigl(1\oplus{\renewcommand{\arraystretch}{0.7}\hspace{-1mm}\begin{array}{c} 2\\ 1\ 1 \end{array}\hspace{-1mm}},0\Bigr) \ar@{-}[r] \ar@{-}[d] &
 \Bigl({\renewcommand{\arraystretch}{0.7}\hspace{-1mm}\begin{array}{c} 2\ 2\\ 1\ 1\ 1 \end{array}\hspace{-1mm}\oplus\hspace{-1mm}\begin{array}{c} 2\\ 1\ 1 \end{array}\hspace{-1mm}},0\Bigr) \ar@{-}[r] &
 \Bigl({\renewcommand{\arraystretch}{0.7}\hspace{-1mm}\begin{array}{c} 2\ 2\\ 1\ 1\ 1 \end{array}\hspace{-1mm}\oplus\hspace{-1mm}\begin{array}{c} 2\ 2\ 2\\ 1\ 1\ 1\ 1 \end{array}\hspace{-1mm}},0\Bigr) \ar@{-}[r] & \cdots
\\
 \Bigl(1,{\renewcommand{\arraystretch}{0.7}\hspace{-1mm}\begin{array}{c} 2\\ 1\ 1 \end{array}\hspace{-1mm}}\Bigr) \ar@{-}[r] &
 \Bigl(0,1\oplus{\renewcommand{\arraystretch}{0.7}\hspace{-1mm}\begin{array}{c} 2\\ 1\ 1 \end{array}\hspace{-1mm}}\Bigr) \ar@{-}[r] &
 \bigl(2,1\bigr) \ar@{-}[r] & \cdots
}
\]
 The corresponding $g$-vectors
\[
\xymatrix@C=12mm@R=5mm{
 (1,0), (0,1) \ar@{-}[r] \ar@{-}[d] & (-1,2), (0,1) \ar@{-}[r] & (-1,2), (-2,3) \ar@{-}[r] & \cdots
\\
 (1,0), (0,-1) \ar@{-}[r] & (-1,0), (0,-1) \ar@{-}[r] & (-1,0), (-2,1) \ar@{-}[r] & \cdots
}
\]
coincide with the $g$-vectors of the corresponding cluster variables as in Subsection \ref{excl}. 


\end{document}